\documentclass[11pt]{amsart}   	
\usepackage{geometry}             
\usepackage[dvipsnames]{xcolor}   		             		
\usepackage{graphicx}				
\usepackage{amssymb,mathrsfs}
\usepackage{amssymb,amsthm,amsmath,stmaryrd}
\SetSymbolFont{stmry}{bold}{U}{stmry}{m}{n}
\usepackage{tikz}
\usepackage{tikz-cd}
\usepackage{accents,upgreek,enumerate}
\usepackage[headings]{fullpage}
\usepackage{bm}
\usepackage[all]{xy}
\usepackage{caption}
\def\rest#1{\big|_{#1}}

\newtheorem{innercustomthm}{Theorem}
\newenvironment{customthm}[1]
{\renewcommand\theinnercustomthm{#1}\innercustomthm}
{\endinnercustomthm}
\usetikzlibrary{calc}
\usetikzlibrary{fadings}
\usetikzlibrary{decorations.pathmorphing}
\usetikzlibrary{decorations.pathreplacing}

\newcounter{marginnote}
\DeclareMathAlphabet{\mathpzc}{OT1}{pzc}{m}{it}

\usepackage[backref]{hyperref}
\hypersetup{
colorlinks   = true,          
urlcolor     = violet,          
linkcolor    = Purple,          
citecolor   = RubineRed             
}

\newtheorem{theorem}{Theorem}[subsection]
\newtheorem{corollary}[theorem]{Corollary}
\newtheorem{lemma}[theorem]{Lemma}
\newtheorem{proposition}[theorem]{Proposition}

\newtheorem{quasi-theorem}[theorem]{Quasi-Theorem}

\theoremstyle{definition}
\newtheorem{definition}[theorem]{Definition}

\newtheorem{guidingprinciple}[theorem]{Guiding Principle}
\newtheorem{construction}[theorem]{Construction}
\newtheorem{example}[theorem]{Example}
\newtheorem{blank remark}[theorem]{}

\theoremstyle{remark}
\newtheorem{rem1}[theorem]{Remark}
\newenvironment{remark}{\begin{rem1}\em}{\end{rem1}}

\newtheorem{not1}[theorem]{Notation}

\newcommand{\A}{{\mathbf{A}}}

\newcommand{\NN} {{\mathbf N}}		
\newcommand{\PP}{\mathbf{P}}         
		
\newcommand{\RR} {{\mathbf R}}		
\newcommand{\ZZ} {{\mathbf Z}}		

\def\R{\mathrm{R}}
\def\L{\mathrm{L}}

\def\LogStr{\mathbf{LogStr}}
\def\Mon{\mathbf{Mon}}
\def\Gm{\mathbf{G}_m}
\def\Proj{\operatorname{\mathbf{Proj}}}
\def\setminus{\smallsetminus}

\def\ctr{{\rm ctr}}

\newcommand{\Hom}{\operatorname{Hom}}

\DeclareMathOperator{\spec}{Spec}

\def\VZ{\mathcal{V\!Z}}
\def\VQ{\mathcal{V\!Q}}

\newcommand{\cal}{\mathcal}

\def\cM{{\cal M}}

\def\fM{\mathfrak{M}}

\newcommand{\plC}{\scalebox{0.8}[1.3]{$\sqsubset$}}

\newcommand{\Mbar}{\overline{\cM}\vphantom{\cM}}

\newcommand{\Spec}{\operatorname{Spec}}

\makeatletter
\def\blfootnote{\xdef\@thefnmark{}\@footnotetext}
\makeatother

\title[Moduli of stable maps in genus one: absolute theory]{{\larger M}oduli of stable maps in genus one {\it \&} logarithmic geometry I}

\date{}

\author[Ranganathan]{Dhruv Ranganathan}
\address{Dhruv Ranganathan \\ Department of Pure Mathematics and Mathematical Statistics\\
University of Cambridge}
\email{\href{mailto:dr508@cam.ac.uk}{dr508@cam.ac.uk}}

\author[Santos-Parker]{Keli Santos-Parker}
\address{Keli Santos-Parker \\ Department of Mathematics\\
University of Colorado}
\email{\href{mailto:keli.parker@colorado.edu}{keli.parker@colorado.edu}}

\author[Wise]{Jonathan Wise}
\address{Jonathan Wise \\ Department of Mathematics\\
University of Colorado}
\email{\href{mailto:jonathan.wise@colorado.edu}{jonathan.wise@colorado.edu}}

\begin{document}

\begin{abstract}
This is the first in a pair of papers developing a framework for the application of logarithmic structures in the study of singular curves of genus $1$. We construct a smooth and proper moduli space dominating the main component of Kontsevich's space of stable genus $1$ maps to projective space. A variation on this theme furnishes a modular interpretation for Vakil and Zinger's famous desingularization of the Kontsevich space of maps in genus $1$. Our methods also lead to smooth and proper moduli spaces of pointed genus $1$ quasimaps to projective space. Finally, we present an application to the log minimal model program for $\Mbar_{1,n}$. We construct explicit factorizations of the rational maps among Smyth's modular compactifications of pointed elliptic curves. 
\end{abstract}

\maketitle

\vspace{-0.2in}

\setcounter{tocdepth}{1}
\tableofcontents

\section{Introduction}


This paper is the first in a pair, exploring the interplay between tropical geometry, logarithmic moduli theory, stable maps, and  moduli spaces of genus~$1$ curves. We focus on the following two applications.

\noindent
{\bf I. Moduli of elliptic curves in $\PP^r$.} We construct a smooth and proper moduli space compactifying the space of maps from pointed genus $1$ curves to $\PP^r$. The natural map to the Kontsevich space is a desingularization of the principal component. A mild variation of this moduli problem yields a modular interpretation for Vakil and Zinger desingularization of the Kontsevich space in genus $1$. We establish analogous results for the space of genus $1$ pointed stable quasimaps to $\PP^r$.

\noindent
{\bf II. Birational geometry of moduli spaces.} The aforementioned application relies on general structure results concerning the geometry of the elliptic $m$-fold point.  We develop techniques to study such singularities using logarithmic methods. This leads to a modular factorization of the birational maps relating Smyth's spaces of pointed genus $1$ curves. 

Blowups of moduli spaces usually do not have modular interpretations. A technical contribution of this work is to demonstrate how tropical techniques allow one to establish modular interpretations for logarithmic blowups of logarithmic moduli spaces, by adding tropical information to the moduli problem.  The concept of minimality -- now standard in logarithmic moduli theory -- returns a corresponding moduli problem on schemes. In the sequel, we extend our results on desingularization to logarithmic targets by constructing toroidal moduli of genus $1$ logarithmic maps to any toric variety~\cite{RSW2}.

\subsection{The main component of genus~$1$ stable maps} The moduli spaces of stable maps in higher genus are essentially never smooth.  For almost all values of $r$ and $d$, the space $\Mbar_{1,n}(\PP^r,d)$ is reducible, not equidimensional, and highly singular.  A remarkable iterated blowup construction due to Vakil and Zinger, however, leads to a smooth moduli space $\widetilde\cM_{1,n}(\PP^r,d)$ compactifying the main component~\cite{VZ07,VZ08}. Hints of the geometry of this resolution are present in Vakil's thesis~\cite[Lemma~5.9]{Vak00}. 

The construction of the space $\widetilde\cM_{1,n}(\PP^r,d)$  is elegant, and it shares many of the excellent properties of $\Mbar_{0,n}(\PP^r,d)$, including smoothness, irreducibility, and normal crossings boundary. However, a closure operation implicit in the construction destroys any natural modular interpretation. As a consequence, the smoothness of $\widetilde\cM_{1,n}(\PP^r,d)$ requires a difficult technical analysis~\cite{HL10,VZ07}, and clouds attempts at conceptual generalizations, for instance into the logarithmic category or to quasimap variants. We first supply a moduli space that desingularizes the main component of $\Mbar_{g,n}(\PP^r,d)$ and then use this perspective to investigate generalizations and related geometries. 

\subsection{Modular desingularization} The central construction of this paper is a moduli space ${\mathfrak M}^{\mathrm{rad}}_{1,n}$ realizing a blowup of the moduli space of genus~$1$, $n$-marked, prestable curves:
\begin{equation*}
{\mathfrak M}^{\mathrm{rad}}_{1,n} \to \mathfrak M_{1,n}
\end{equation*}
This blowup parameterizes prestable curves $C$ equipped with a \textbf{radial alignment} of their tropicalizations $\plC$ -- this may be thought of as a total ordering on the vertices of the dual graph $\plC$ of $C$ that are not members of the smallest subcurve of genus $1$. We emphasize that this is an algebraic stack over schemes. See Section~\ref{sec: strategy} and Section~\ref{sec:aligned}.

Given a stable map $[f : C \to Y]$, the radial alignment determines both a semistable model $\widetilde C$ of $C$, and a projection $\tau : \widetilde C \to \overline C$ that contracts a genus~$1$ subcurve of $\widetilde C$ to a genus~$1$ singularity.  

\begin{customthm}{A}\label{thm: ordinary}
Let $Y$ be a smooth and proper complex variety and fix a curve class $\beta\in H_2(Y, \mathbf Z)$.  Consider the following data as a moduli problem over schemes:
\begin{enumerate}
\item a minimal family of $n$-marked, radially aligned, logarithmic curves, $C \to S$,
\item a stable map $f : C \to Y$ such that $f_\star[C] = \beta$, and
\item a factorization of $\widetilde C \to C \xrightarrow{f} \PP^r$ through the canonical contraction $\widetilde C \to \overline C$ that is nonconstant on a branch of the central genus~$1$ component of $\overline C$.
\end{enumerate}
This moduli problem is represented by a proper Deligne--Mumford stack $\VZ_{1,n}(Y, \beta)$, carrying a natural perfect obstruction theory.  The space $\VZ_{1,n}(\PP^r,d)$ is smooth and irreducible of expected dimension.
\end{customthm}

It is natural to wonder how the Vakil--Zinger blowup construction relates to $\VZ_{1,n}(\PP^r,d)$. The relationship arises via the concept of a \textbf{central alignment}, which can be thought of as a partial ordering of the vertices, whereas the radial alignment is total. 

\begin{customthm}{B}\label{thm: vz-space}
There exists a proper Deligne--Mumford stack $\VZ^\ctr_{1,n}(Y,\beta)$ parameterizing stable maps from minimal families of centrally aligned genus $1$, $n$-pointed curves to $Y$, satisfying the factorization property. When $Y = \PP^r$ there is an isomorphism
\[
\VZ^\ctr_{1,n}(\PP^r,d)\to \widetilde{\cM}_{1,n}(\PP^r,d).
\]
\end{customthm}


\subsection{The quasimap moduli} When there are no marked points on the source curve, there is an alternate non-singular compactification to $\VZ_{1}(\PP^r, d)$ via the theory of stable quasimaps, also called stable quotients, see~\cite{CFK,MOP}. Rather than a blowup of $\Mbar_1(\PP^r,d)$, the quasimap space $\mathcal{Q}_1(\PP^r,d)$ is a contraction, fitting into a diagram
\[
\VZ_1(\PP^r,\beta)\to \Mbar_1(\PP^r,d)\to \mathcal Q_1(\PP^r,d).
\]
In this sense, the stable quasimap spaces are efficient compactifications, giving one point of access to the geometry of elliptic curves in $\PP^r$. When marked points are present, the stable quotient spaces are no longer smooth, and can be essentially as singular as the space of maps. 

We desingularize the pointed spaces using radial alignments. As before, a radially aligned curve $C$ equipped with a quasimap to $\PP^r$ produces a semistable model $\widetilde C$ of $C$ and a contraction $\widetilde C \to \overline C$ of the genus~$1$ component.

\begin{customthm}{C}\label{thm: quotients}
Fix a degree $d$. 
Consider the following data as a moduli problem on schemes:
\begin{enumerate}
\item a minimal family of $n$-marked, radially aligned, logarithmic curves, $C \to S$, and
\item a stable quasimap $f$ from $C$ to $\PP^r$ of degree $d$, such that
\item $f$ factors through a quasimap $\overline C \to \PP^r$ having positive degree on at least one branch of the genus~$1$ component.
\end{enumerate}
This moduli problem is represented by a smooth, proper Deligne--Mumford stack $\VQ_{1,n}(\PP^r, d)$ of the expected dimension.
\end{customthm}

In both stable map and quasimap theories, smooth is proved conceptually, without a local analysis of the singularities of the ordinary moduli spaces, which is the core of previous approaches to the problem.

\subsection{Elliptic singularities {\it \&} logarithmic geometry} For each integer $m\geq 1$, the elliptic $m$-fold point is the unique Gorenstein genus $1$ singularity with $m$ branches, see Section~\ref{sec: genus-1-singularities}. For each $m$, Smyth constructs a proper and irreducible moduli space $\Mbar_{1,n}(m)$ of curves with elliptic $\ell$-fold singularities, for $\ell\leq m$ and an appropriate global stability condition. However, the spaces are smooth if and only if $m\leq 5$. By the irreducibility, for each $m$, there is a rational map
\[
\Mbar_{1,n}\dashrightarrow \Mbar_{1,n}(m).
\]
We construct a factorization of this rational map by building a single smooth moduli space that maps to both, via operations on its universal curve. 

\begin{customthm}{D}
Let $\Mbar_{1,n}^{\mathrm{rad}}$ denote the moduli space of radially aligned $n$-pointed genus $1$ curves. There is a canonical factorization of the rational map $\Mbar_{1,n}\dashrightarrow \Mbar_{1,n}(m)$ as
\[
\begin{tikzcd}
& \Mbar_{1,n}^{\mathrm{rad}} \arrow{dr}{\phi_m} \arrow[swap]{dl}{\pi} & \\
\Mbar_{1,n}\arrow[dashrightarrow]{rr} & & \Mbar_{1,n}(m).
\end{tikzcd}
\]
The map $\pi$ is a logarithmic blowup, while the map $\phi_m$ induces a contraction of the universal curve of $\Mbar_{1,n}^{\mathrm{rad}}$.
\end{customthm}

The space $\Mbar_{1,n}^{\mathrm{rad}}$ has the best properties of both spaces in the lower part of the diagram -- it is smooth with a normal crossings boundary, the boundary combinatorics is explicit, and it sees the geometry of elliptic $m$-fold singular curves. 

\subsection{Previous work on genus $1$ maps} There has been a substantial amount of work on the moduli space of genus $1$ stable maps to $\PP^r$ in the last decade, which we can only summarize briefly. The seminal application of the Vakil--Zinger desingularization was to the proof Bershadsky-Cecotti-Ooguri-Vafa's prediction for the genus $1$ Gromov--Witten invariants of Calabi--Yau hypersurfaces~\cite{Z09}. The desingularization was revisited by Hu and Li, who provided a different perspective on the blowup construction~\cite{HL10}. While the techniques in the present text handle arbitrary proper algebraic targets, there is a ``sharp Gromov compactness'' result for arbitrary K\"ahler targets using symplectic Gromov--Witten theory by work of Zinger~\cite{Z09b}. It would be interesting to develop a modular interpretation, as we do here, for K\"ahler and symplectic targets. We imagine that our methods would work equally well for logarithmic analytic spaces, but Parker's category of exploded manifolds may already contain the essential ingredients~\cite{Par12a}.  

The situation is simpler in the absence of marked points. The theories of stable quotients and quasimaps, due to Marian--Oprea--Pandharipande and Ciocan-Fontanine--Kim, provide smooth and proper moduli of genus $1$ curves in $\PP^r$ with no marked points~\cite{CFK,MOP}. These spaces have a beautiful geometry -- Cooper uses the modular interpretation to show that $\mathcal Q_1(\PP^r,d)$ is rationally connected with Picard number $2$, explicitly computes the canonical divisor, and gives generators for the Picard group~\cite{Coop15}. It would be natural to use the desingularization here to extend Cooper's study to the pointed space. Kim's proposal of maps to logarithmic expansions also produces a nonsingular moduli space of maps to $\mathbf P^r$ relative to a smooth divisor, provided there are no ordinary or relative marked points~\cite{Kim08}. 

A different direction was pursued in an elegant paper of Viscardi~\cite{Vis}, who extended Smyth's construction to the setting of maps. The resulting spaces $\Mbar_{1,n}^{(m)}(Y,d)$ are proper, smooth when all numerical parameters are small, and irreducible when $m$ is large. In fact, for $m\gg 0$, the space is smooth over the singular Artin stack $\fM_{1,n}(m)$ parameterizing genus $1$ curves with at worst elliptic $m$-fold singularities, and thus, in spirit, his approach is close to ours. Crucially, however, our base moduli space of radially aligned curves has a better deformation theory, so that the moduli space is smooth when $Y = \PP^r$ and not merely relatively smooth over a non-smooth base.

\subsection{User's guide} The central technical result of this paper is the construction of the moduli space of prestable radially aligned genus $1$ curves in Section~\ref{sec:aligned}. The corresponding moduli space of stable objects is related to Smyth's space via a contraction of the universal curve in Theorem~\ref{thm: contraction to smyth}. The space $\VZ_{1,n}(Y,\beta)$ is constructed in Section~\ref{sec: vz-construction}, shown to be proper in Theorem~\ref{vz-properness}, and to have a virtual class in Theorem~\ref{thm: vfc}. The non-singularity for target $\PP^r$ is then established in Theorem~\ref{thm: vz-smoothness} via deformation theory, and the comparison with Vakil--Zinger's construction is undertaken in Section~\ref{vz-comparison}. We desingularize the quasimaps spaces in Section~\ref{sec: quasimaps}. 

\subsection*{Acknowledgements} Thanks are due to Dan Abramovich, Dori Bejleri, Sebastian Bozlee, Renzo Cavalieri, Dave Jensen, Eric Katz, Diane Maclagan, Davesh Maulik, Sam Payne, David Speyer, Ravi Vakil, and Jeremy Usatine for helpful conversations and much encouragement. The authors learned that they were working on related ideas after a seminar talk of D.R. at Colorado State University when J.W. was in the audience; we thank Renzo Cavalieri for creating that opportunity. We are grateful to the anonymous referees, whose suggestions led to an improved text.

\subsection*{Funding} D.R. was partially supported by NSF grant DMS-1128155 (Institute for Advanced Study) and J.W. was partially supported by NSA Young Investigator's Grants H98230-14-1-0107 and H98230-16-1-0329.

\section{Preliminaries}

\subsection{Genus 1 singularities}\label{sec: genus-1-singularities} 
Let $C$ be a reduced curve over an algebraically closed field $k$ and let $(C,p)$ be an isolated singularity. There are two basic invariants of this singularity. Let 
\[
\pi: (\widetilde C,p_1,\ldots, p_m)\to (C,p)
\]
be the normalization, where $\{p_i\}$ is the inverse image of $p$. The number $m$ is referred to as the \textbf{number of branches of the singularity}. The second invariant, the \textbf{$\delta$-invariant}, is defined by
\[
\delta: = \dim_k \bigl( \pi_\star (\mathscr O_{\widetilde C})/\mathscr O_C \bigr).
\]

Let $\mathscr A \subset \pi_\star \mathscr O_{\widetilde C}$ be the subring of functions that are well-defined on the underlying topological space of $C$.  In a neighborhood of a point $p$ of $C$, the ring $\mathscr A$ can be constructed as a fiber product:
\begin{equation*}
\pi_\star \mathscr O_{\widetilde C} \mathop\times_{\pi_\star \mathscr O_{\pi^{-1} (p)}} \mathscr O_p
\end{equation*}
Then $\mathscr A$ is the structure sheaf of a scheme, called the \textbf{seminormalization} of $C$. 

\begin{definition}
The \textbf{genus} of a singularity $(C,p)$ is the quantity
\begin{equation*}
g = \dim_k \bigl( \mathscr A / \mathscr O_C \bigr) 
\end{equation*}
where $\mathscr A$ is the structure sheaf of the seminormalization of $C$.
\end{definition}

By construction, we have
\[
g = \delta-m+1.
\]

The term genus is consistent with the usual notion of arithmetic genus:  if $C$ is proper (so that the arithmetic genus is well-defined), its arithmetic genus differs from the genus of its seminormalization by $g$.  Alternatively, the stable reduction of a $1$-parameter smoothing of $C$ replaces $p$ with a nodal curve of arithmetic genus $g$.

We will be concerned with singularities of genus $1$ in this paper.

\begin{proposition}
There is a unique Gorenstein singularity of genus $1$ with $m$ branches. Specifically, if $m = 1$ the singularity is the cusp $\mathbf{V}(y^2-x^3)$, if $m = 2$ the singularity is the ordinary tacnode $\mathbf V(y^2-yx^2)$, and for $m\geq 3$, the singularity is the union of $m$ general lines through the origin in $\mathbf A^{m-1}$. 
\end{proposition}

\begin{proof}
See~\cite[Proposition~A.3]{Smyth}.
\end{proof}

\begin{proposition} \label{prop:dualizing-trivial}
The dualizing sheaf of a Gorenstein curve of genus~$1$ with no genus~$0$ tails is trivial.
\end{proposition}

\begin{proof}
Let $C$ be a Gorenstein, genus~$1$ curve with no genus~$0$ tails.  Then $C$ is either smooth, a ring of rational curves, or an elliptic $m$-fold point.  If $C$ is smooth then, $\omega_C$ has degree zero and has a nonzero global section, hence is trivial.  If $C$ is a ring of rational curves, then $\omega_C$ restricts to have degree zero on each component, yet has a nonzero global section, hence is trivial.  Finally, if $C$ is an elliptic $m$-fold point then a local calculation shows that $\omega_C$ restricts to $\omega_{C_i}(2) \simeq \mathscr O_{C_i}$ for each rational component $C_i$ of $C$.  One can then find explicit local generators for $\omega_C$ that extend globally. Such generators are, for instance, recorded in~\cite[Proposition~2.1.1]{RSW2}.
\end{proof}

\begin{corollary} \label{cor:twisted-dualizing}
	Suppose that $C$ is a connected semistable curve of genus~$1$ curve with a nonempty collection of marked points $p_1, \ldots, p_n$ and let $\Sigma = \sum p_i$.  Then $H^1(C, \omega_C^{\otimes k}(k\Sigma)) = 0$.
\end{corollary}
\begin{proof}
	Let $C_0$ be the circuit component (or union of components) of $C$ and $C_i$ the remaining components.  The dual graph of the $C_i$ is a tree, so
	\begin{equation*}
		H^1(C, \omega_C^{\otimes k}(k\Sigma)) = \sum H^1(C_i, \omega_C^{\otimes k}(k \Sigma)\big|_{C_i}) .
	\end{equation*}
	If $i \neq 0$ then $C_i$ is rational and by the semistability of $C$ there are at least two marked points or nodes on $C_i$.  Therefore $\omega_C^{\otimes k}(k\Sigma) \big|_{C_i}$ has non-negative degree and hence also vanishing $H^1$.  On $C_0$, we can identify $\omega_C(\Sigma) \big|_{C_0}$ with $\omega_{C_0}(\sum q_i)$, where the $q_i$ are the external nodes and marked points of $C_0$.  By Proposition~\ref{prop:dualizing-trivial}, we know that $\omega_{C_0}$ is trivial, so $H^1(C_0, \omega_C^{\otimes k}(k \Sigma)\big|_{C_0})$ is dual to $H^0(C_0, \mathscr O_{C_0}(-k \sum q_i))$.  But the only sections of $\mathscr O_{C_0}$ are constants, and there is at least one $q_i$, so $H^0(C_0, \mathscr O_{C_0}(-k \sum q_i)) = 0$.
\end{proof}

\subsection{Tropical curves}
We follow the presentation of tropical curves from~\cite[Sections 3 {\it \&} 4]{CCUW}, introducing families of tropical curves. We refer the reader to loc. cit. for a more detailed presentation. 

\begin{definition}
A \textbf{pre-stable $n$-marked tropical curve} $\plC$ is a finite graph $G$ with vertex and edge sets $V$ and $E$, enhanced with the following data
\begin{enumerate}
\item a \textbf{marking function} $m: \{1,\ldots,n\}\to V$,
\item a \textbf{genus function} $g:V\to \NN$,
\item a \textbf{length function} $\ell: E\to \RR_{+}$.
\end{enumerate}
Such a curve is said to be a \textbf{stable $n$-marked tropical curve} if (1) at every vertex $v$ with $g(v) = 0$, the valence of $v$ (including the markings) is at least $3$, and (2) at every vertex $v$ with $g(v) = 1$, the valence of $v$ (including the markings) is at least $1$. The \textbf{genus} of a tropical curve $\plC$ is the sum 
\[
g(\plC) = h_1(G)+\sum_{v\in V} g(v)
\]
where $h_1(G)$ is the first Betti number of the graph $G$.
\end{definition}

In practice, we will intentionally confuse a tropical curve $\plC$ with its geometric realization --- a metric space on the topological realization of $G$, such that an edge $e$ is metrized to have length $\ell(e)$ and if $m(i) = v$, we attached the ray $\RR_{\geq 0}$ to the vertex $v$, as a half-edge with unbounded edge length.

More generally, one may permit the length function $\ell$ above to take values in an arbitrary toric monoic $P$. This presents us with a natural notion of a family of tropical curves.

\begin{definition}
Let $\sigma$ be a rational polyhedral cone with dual monoid $S_\sigma$ (the integral points of its dual cone). A \textbf{family of $n$-marked prestable tropical curves over $\sigma$} is a tropical curve whose length function takes values in $S_\sigma\setminus \{0\}$.
\end{definition}

To see that such an object is, in an intuitive sense, a family of tropical curves, observe that the points of $\sigma$ can be identified with monoid homomorphisms
\[
\varphi: S_\sigma \to \RR_{\geq 0}.
\]
Given such a homomorphism $\varphi$ and an edge $e\in E$, the quantity $\varphi(\ell(e))$ is an ``honest'' length for $e\in E$. The resulting tropical curve can be thought of as the fiber of the family over $[\varphi]\in \sigma$.

\subsection{Logarithmic {\it \&} tropical curves}\label{sec: log-trop}
Let $(S,M_S)$ be a logarithmic scheme. A \textbf{family of logarithmically smooth curves over $S$} (or \textbf{logarithmic curve over $S$} for short) is a logarithmically smooth and proper morphism
\[
\pi: (C,M_C) \to (S,M_S),
\]
of logarithmic schemes with $1$-dimensional connected fibers with two additional technical conditions: $\pi$ is required to be integral and saturated. These are conditions on the morphism $\pi^\flat: M_S\to M_C$ that guarantee that $\pi$ is flat with reduced fibers. The \'etale local structure theorem for such curves, due to F. Kato, characterizes such families locally on the source~\cite{Kat00}.  We write $\mathfrak M_{g,n}^{\log}$ for the stack of families of connected, proper, $n$-marked, genus~$g$ families of logarithmic curves over logarithmic schemes.

\begin{theorem}
Let $C\to S$ be a family of logarithmically smooth curves. If $x\in C$ is a geometric point, then there is an \'etale neighborhood of $C$ over $S$, with a strict morphism to an \'etale-local model $\pi:V\to S$, and $V\to S$ is one of the following:
\begin{itemize}
\item { (the smooth germ)} $V  = \mathbf A^1_S\to S$, and the logarithmic structure on $V$ is pulled back from the base;
\item { (the germ of a marked point)} $V = \mathbf A^1_S\to S$, with logarithmic structure pulled back from the toric logarithmic structure on $\mathbf A^1$;
\item { (the node)} $V = \mathscr O_S[x,y]/(xy = t)$, for $t\in \mathscr O_S$.  The logarithmic structure on $V$ is pulled back from the multiplication map $\mathbf A^2 \to \mathbf A^1$ of toric varieties along a morphism $t : S \to \mathbf A^1$ of logarithmic schemes.
\end{itemize}

The image of $t\in M_S$ in $\overline M_S$ is called the \textbf{deformation parameter of the node}.


\end{theorem}

Associated to a logarithmic curve $C\to S$ is a family of tropical curves.  As the construction is simpler when the underlying scheme of $S$ is the spectrum of an algebraically closed field, and we will only need it in that case, we make that assumption in order to describe it.  Under this assumption, for each edge $e$ of the dual graph of $C$, we write $\delta_e$ for the deformation parameter of the corresponding node of $C$.  The following definition is given implicitly by Gross and Siebert~\cite[Section 1]{GS13}:

\begin{definition}\label{def:tropicalization}
Let $C$ be a logarithmic curve over $S$, where the underlying scheme of $S$ is the spectrum of an algebraically closed field.  The \textbf{tropicalization} of $C$ is the dual graph $\plC$ of $C$, with vertices weighted by the genera of the corresponding components of $C$, and with the length of an edge $e$ defined to be the smoothing parameter $\delta_e \in \overline M_S$.
\end{definition}




\subsection{Line bundles from piecewise linear functions} \label{sec:pwl} 

It is shown in \cite[Remark~7.3]{CCUW} that, if $C$ is a logarithmic curve over $S$, and the underlying scheme of $S$ is the spectrum of an algebraically closed field, then sections of $\overline M_C$ may be interpreted as piecewise linear functions on the tropicalization of $C$ that are valued in $\overline M_S$ and are linear along the edges with integer slopes.

For any logarithmic scheme $X$ and any section $\alpha \in \Gamma(X, \overline M_X^{\rm gp})$, the image of $\alpha$ under the coboundary map
\begin{equation*}
H^0(X, \overline M_X^{\rm gp}) \to H^1(X, \mathscr O_X^\star)
\end{equation*}
induced from the short exact sequence
\begin{equation*}
0 \to \mathscr O_X^\star \to M_X^{\rm gp} \to \overline M_X^{\rm gp} \to 0
\end{equation*}
represents an $\mathscr O_X^\star$-torsor $\mathscr O_X^\star(-\alpha)$ on $X$.  Via the equivalence between $\mathscr O_X^\star$-torsors and line bundles, this corresponds to a line bundle, $\mathscr O_X(-\alpha)$. To each piecewise linear function $f$ on $\plC$ that is linear on the edges with integer slopes and takes values in $\overline M_S$, we have an associated section of $\overline M_C$ and therefore an associated line bundle $\mathscr O(-f)$.

The monoid $\overline M_X \subset \overline M_X^{\rm gp}$ gives $\overline M_X^{\rm gp}$ a partial order in which $f \leq g$ when $g - f \in \overline M_X$.  If $f \geq 0$, meaning that $f$ is a section of $\overline M_X$, then we can restrict $\varepsilon : M_X \to \mathscr O_X$ to give a homomorphism $\mathscr O_X(-f) \to \mathscr O_X$.  More generally, if $f \leq g$ then $g - f \geq 0$ and we get $\mathscr O_X(f - g) \to \mathscr O_X$, hence $\mathscr O_X(f) \to \mathscr O_X(g)$.

A logarithmic structure can be defined equivalenty as a system of invertible sheaves indexed homomorphically by the sheaf of partially ordered abelian groups $\overline M_X^{\rm gp}$.  We will frequently take this point of view in the sequel. The line bundles and torsors arising from the logarithmic structure on a curve can also be described in a rather explicit fashion using chip-firing and tropical divisor theory. We refer the reader to~\cite{FRTU,Kaj} for developments in this direction. 

\begin{proposition} \label{prop:line-bundle-degrees}
Let $\pi : C \to S$ be a logarithmic curve over $S$.  Assume that $\overline M_S$ and the dual graph $\plC$ are constant over $S$ and that the smoothing parameters of the nodes are all zero in $\mathscr O_S$.  If $f$ is a piecewise linear function on $\plC$ that is linear with integer slopes on the edges and takes values in $\overline M_S$, and $C_v$ is the component of $C$ corresponding to the vertex $v$ of $\plC$, then there is a canonical identification
\begin{equation*}
\mathscr O_C(f) \big|_{C_v} = \mathscr O_{C_v}\Bigl(\sum_e \mu_e p_e\Bigr) \otimes \pi^\star \mathscr O_S(f(v))
\end{equation*}
where the sum is taken over flags $e$ of $\plC$ rooted at $v$, the integer $\mu_e$ is the outgoing slope of $f$ along the edge $e$, and $p_e$ is the point of $C_v$ corresponding to $e$.
\end{proposition}
\begin{proof}
If $f$ is a constant function then the statement is obvious, and both sides of the equality are additive functions of $f$, so we may subtract the constant function with value $f(v)$ from $f$ and assume that $f(v) = 0$.  Let $C_v^\circ$ be the interior of $C_v$.  As $f$, viewed as a section of $\overline M_C$, takes the constant value $0$ on $C_v^\circ$, there is a canonical trivialization of $\mathscr O_C(-f)$ on $C_v^\circ$.

Consider an edge $e$ of $\plC$ that is incident to $v$.  This corresponds to a node $p$ of $C$ that lies on $C_v$ with local coordinates $\alpha + \beta = \delta$, with $\alpha, \beta \in \overline M_{C,p}$ and $\delta \in \overline M_S$.  Let $\tilde\alpha$ and $\tilde\beta$ be lifts of $\alpha$ and $\beta$ to $M_{C,p}$.  Either $\varepsilon(\tilde\alpha)$ or $\varepsilon(\tilde\beta)$ restricts to a local parameter $C_v$ at $p$, so we assume without loss of generality that it is $\varepsilon(\tilde\alpha)$.  

If the slope of $f$ along $e$ is $m$ then $f$ corresponds locally to $m \alpha$.  We assume first that $m \geq 0$.  Then $\varepsilon$ restricts on a neighborhood $U$ of $p$ in $C_v$ to give
\begin{equation*}
\varepsilon \big|_{U} : \mathscr O_U(-f) \to \mathscr O_{U}
\end{equation*}
whose image is the ideal generated by $x^m$.  This gives a \emph{canonical} isomorphism between $\mathscr O_U(-f)$ and $\mathscr O_U(-mp)$ in a neighborhood $U$ of $p$ that restricts on the complement of $p$ to the trivialization described above. If $m < 0$ then $-m \geq 0$ and we obtain a canonical isomorphism $\mathscr O_U(f) \simeq \mathscr O_U(-mp)$ in a neighborhood $U$ of $p$, as above.  This completes the proof.
\end{proof}

\subsection{Logarithmic blowups and logarithmic modifications}
\label{sec:log-blowup}

Let $X$ be a logarithmic scheme and let $I \subset \overline M_X$ be a coherent ideal, by which we mean that $I$ is a subsheaf of $\overline M_X$ such that $\overline M_X + I = I$ and locally $I$ is generated by global sections of $\overline M_X$ (see \cite[Definition~3.6]{Kato-LogMod}).  We say $I$ is principal if it is possible to find a section $\alpha$ of $\overline M_X$ such that $I = \alpha + \overline M_X$.  Note that this is actually a local condition, as $\alpha$ is unique if it exists because $\overline M_X$ is sharp.

Given any ideal $I \subset \overline M_X$, and a logarithmic scheme $S$, we define $F(S)$ to be the set of logarithmic maps $f : S \to X$ such that $f^\star I$ is principal.

Suppose that $I$ is generated by sections $\alpha_j$.  Then $F(S)$ is, equivalently, the set of logarithmic maps $f : S \to X$ such that the collection $\{ f^\star (\alpha_j) \}$  of sections of $\overline M_S$ has a minimal element with respect to the partial order introduced in Section~\ref{sec:pwl}.  This interpretation will be useful when we relate the Vakil--Zinger blowup construction to our own in Section~\ref{sec: vz-construction}.

\begin{proposition}
The functor $F$ is representable by a logarithmic scheme, called the logarithmic blowup of $I$.
\end{proposition}
\begin{proof}
	The assertion is local in the \'etale topology, so we can assume $X$ has a global chart, which we regard as a strict map to a toric variety $X \to V$.  Then $F$ is the base change of the moduli problem over $V$ defined by the same formula, so we can assume $X = V$ is a toric variety.  Then $I$ generates a toric ideal of $X$ and the blowup of that ideal, in the usual sense, represents $F$.  See the discussion following Definition~3.8 in \cite{Kato-LogMod} for a more detailed construction.
\end{proof}


\begin{remark}\label{rem:mono}
Let $\tilde X$ be a logarithmic blowup of $X$. It may be counterintuitive that although $\tilde X\to X$ is essentially never an injection, the functor on logarithmic schemes defined by $\tilde X$ is defined as a subfunctor of the one defined by $X$. That is, a logarithmic blowup is a non-injective monomorphism. This may be seen as a failure of the schemes $\tilde X$ and $X$ to be good topological reflections of the associated logarithmic schemes. An artifact of the monomorphicity is reflected in the fact that the map at the level of tropicalizations (cone complexes) is a set theoretic bijection. 
\end{remark}

\begin{remark}\label{rem:mono-blowup}
The monomorphicity of logarithmic blowups might be understood by comparison with the conventional universal property for blowing up in algebraic geometry \cite[Proposition~II.7.14]{Hartshorne}, which also asserts that $\Hom(S, \tilde X) \to \Hom(S, X)$ is injective for morphisms $S \to X$ that meet the blowup center sufficiently transversely.  Logarithmic geometry forces all morphisms meeting the logarithmic boundary of $X$ to ``know something about'' points of $X$ nearby the boundary, effectively making all logarithmic morphisms $S \to X$ sufficiently transverse.
\end{remark}

Of particular interest in this paper will be the logarithmic blowups that arise from ideals with $2$ global generators, $\alpha$ and $\beta$ in $\Gamma(X, \overline M_X)$.  Then the blowup $F$ constructed above is the universal $Y \to X$ such that the restrictions of $\alpha$ and $\beta$ are \emph{locally comparable}.  That is, for every geometric point $y$ of $Y$, we have either $\alpha_y \leq \beta_y$ or $\alpha_y \geq \beta_y$ in the stalk $\overline M_{Y,y}$.

\begin{definition} \label{def:log-mod}
	A morphism of logarithmic schemes $f : X \to Y$ is called a \emph{logarithmic modification} if, locally in $Y$, it is the base change of a toric modification of toric varieties.
\end{definition}

Logarithmic blowups are logarithmic modifications, but not every logarithmic modification is a logarithmic blowup, even locally, because not all toric modifications are toric blowups.  Nevertheless, every logarithmic modification can be dominated by a logarithmic blowup.  We omit an explanation of this fact, since we will not need to make any use of it. 

\begin{example} \label{ex:log-blowup}
In order to make clear how the imposition of an order between a priori unordered elements of $\overline M_X$ translates into a blowup, we work out a basic example.  We assume that $X$ is the spectrum of an algebraically closed field and that $\overline M_X = \mathbf N \alpha + \mathbf N \beta$.  Let $\tilde X$ be the universal logarithmic scheme over $X$ such that $\alpha$ and $\beta$ pull back to comparable elements. We suppress the pullback in what follows.

Of particular interest are the points of $\tilde X$ where $\alpha = \beta$.  Considering only characteristic monoids, it might seem that there is just one such point.  However, to lift from the characteristic sheaf to a logarithmic point, consideration of the logarithmic structure sheaf reveals that these points each require a choice of element in $\mathcal O_X^\ast$ to identify the torsors corresponding to $\alpha$ and $\beta$.  This is the interior of the exceptional locus of the blowup, as we now explain in more detail.

Let $Y$ be the spectrum of an algebraically closed field, with $\overline M_Y = \mathbf N$.  Consider the morphisms $Y \to \tilde X$ that send both $\alpha$ and $\beta$ to $1 \in \mathbf N$.  To produce such a map, we must give a morphism of logarithmic structures $M_X \to M_Y$, which induces (and is determined by) isomorphisms $\mathcal O_X(\alpha) \simeq \mathcal O_Y(1)$ and $\mathcal O_X(\beta) \simeq \mathcal O_Y(1)$.  The ratio of these two isomorphisms gives a well-determined element of $\mathcal O_X^\ast$, from which $Y$ and the map $Y \to \tilde X$ can be recovered up to unique isomorphism.

Put another way, to construct a logarithmic structure $M_Y$ and morphism $M_X \to M_Y$ such that $\alpha$ and $\beta$ are identified in $\overline M_Y$ requires the identification of the invertible sheaves $\mathcal O_X(\alpha)$ and $\mathcal O_X(\beta)$ and there is a $\Gm$-torsor of such identifications available to choose from.

To construct the logarithmic blowup, one may proceed by building two charts, where $\alpha \leq \beta$ and where $\alpha \geq \beta$.  We construct the former.  Take $\overline M_U$ to be the submonoid of $\overline M_X^{\rm gp}$ generated by $\overline M_X$ and by $\beta - \alpha$.  Let $M_U$ be the preimage of $\overline M_U$ in $\overline M_X^{\rm gp}$.  There is now a choice for the map $\varepsilon : M_U \to \mathcal O_X$.  The universal option is to take $\mathcal O_U = \mathcal O_X[z]$ and impose the (vacuous) relation $\varepsilon(\beta) z = \varepsilon(\alpha)$, so that $\varepsilon(\tilde\beta \tilde\alpha^{-1}) = z$ becomes well-defined (for some choice of lifts $\tilde\alpha$ and $\tilde\beta$ of $\alpha$ and $\beta$ to $M_X^{\rm gp}$.  The underlying scheme of $U$ is therefore $\mathbf A^1$.

While there is not a unique choice for $\varepsilon : M_U \to \mathcal O_X$ in the description above, there is a somewhat canonical one, in which $\varepsilon(\tilde\beta \tilde\alpha^{-1}) = 0$.  This is the locus where $\alpha < \beta$, strictly, and corresponds to the origin of the chart $U \simeq \mathbf A^1$. The other chart yields the same result, with the identification giving rise to the logarithmic blowup $\mathbf P^1$.
\end{example}

\section{Moduli spaces of genus one curves}\label{sec: target-point}

The results in this section were obtained in the doctoral dissertation of the second author~\cite{KSP-Thesis}.  Several variants of the main construction of this paper, which are either treated briefly here, or not at all, are described in greater detail in~\cite{KSP-Thesis}.

We construct a moduli space $\Mbar_{1,n}^{\mathrm{rad}}$ of pointed curves with a \textbf{radial alignment}, show that it is a blowup of $\Mbar_{1,n}$, and that the radial alignments determine contraction morphisms to the space of $m$-stable curves, as defined by Smyth~\cite{Smyth}. 

\subsection{The intuition {\it \&} strategy}\label{sec: strategy} The framework in this section may be unintuitive at first, so we provide some motivation that will become precise in later sections. For each integer $m\geq 0$, Smyth constructs proper, not necessarily smooth moduli spaces $\Mbar_{1,n}(m)$ of $m$-stable curves. Here, for each $m$, one considers the moduli problem for curves of arithmetic genus $1$ where the central genus~$1$ component has a total of more than $m$ markings and external nodes (meaning nodes where it meets the complementary subcurve).  In place of the genus~$1$ curves with $m$ or fewer branches, Smyth substitutes Gorenstein genus~$1$ singularities (Section~\ref{sec: genus-1-singularities}).  These spaces are all birational to one another, and there is a birational map identifying the loci of smooth elliptic curves with distinct markings
\[
\Mbar_{1,n}\dashrightarrow \Mbar_{1,n}(m).
\]
The main result of this section is the construction of a moduli space $\Mbar_{1,n}^{\mathrm{rad}}$ that, for any $0\leq m \leq n$ resolves the indeterminacies of the rational map above, i.e.,
\[
\begin{tikzcd}
& \Mbar_{1,n}^{\mathrm{rad}} \arrow{dl} \arrow{dr}{\phi_m}& \\
\Mbar_{1,n} \arrow[dashrightarrow]{rr} & & \Mbar_{1,n}(m).
\end{tikzcd}
\]
We construct this stack by adding information to the moduli problem of $\Mbar_{1,n}$ guided by the following observation: 

\begin{center}
\parbox{.5\textwidth}{
	\textit{An elliptic $m$-fold singularity is formed by contracting a genus $1$ component with $m$ external nodes in a smoothing family.}
}
\end{center}

For example, suppose that $C \to S$ is a $1$-parameter smoothing of a nodal curve $C_0$ with smooth total space and that $E$ is an irreducible genus~$1$ component of $C_0$.  Suppose that $\overline C$ is a flat family obtained from $C$ by contracting $E$.  If $E$ is a genus~$1$ tail then the constancy of the Hilbert polynomial in flat families forces it to be replaced in $\overline C_0$ by a genus~$1$ singularity with one branch --- a cusp.  If $E$ is a genus~$1$ bridge then, assuming $\overline C$ is Gorenstein, the replacement of $E$ will be a tacnode.  

One must take care that, if $m > 1$, then the resulting singularity will have moduli and can depend on the choice of smoothing family.  Therefore the rational map above has indeterminacy.

We mimic the contraction tropically in the following manner. The \textbf{circuit} of a tropical curve of genus $1$ is the union of the vertices whose complement contains no component of genus~$1$.  Given a tropical curve $\plC$ of genus $1$, we may consider the circle around the circuit of radius $\delta^m$, which is the smallest radius in the characteristic monoid of the base such that there are at most $m$ paths from the circuit to the circle, and strictly more than $m$ paths from the circle to infinity; see Figure~\ref{fig: circle}. Contracting the interior of the circle in a family of curves with tropicalization $\plC$ produces an $m$-stable curve.

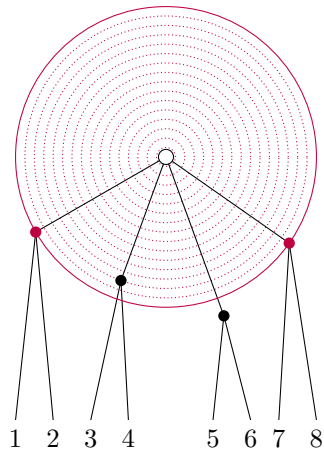
\begin{figure}[ht!]
\begin{tikzpicture}
\foreach \r in {0.125,0.25,...,1.875}
\draw [densely dotted, color=purple] (0,0) circle (\r);

\draw (0,0)--(210:2); \draw (0,0)--(325:2); \draw (0,0)--(250:1.75); \draw (0,0)--(290:2.25);
\fill [color=white] (0,0) circle (1mm);
\draw (0,0) circle (1.001mm);

\draw (210:2)--(-2,-3.5); \draw (210:2)--(-1.5,-3.5);
\draw (250:1.75)--(-1,-3.5); \draw (250:1.75)--(-0.5,-3.5);
\draw (290:2.25)--(0.625,-3.5); \draw (290:2.25)--(1.125,-3.5);
\draw (325:2)--(2,-3.5); \draw (325:2)--(1.5,-3.5);

\draw[color=purple] (0,0) circle (2);
\fill [color=purple] (210:2) circle (0.75mm);
\fill [color=purple] (325:2) circle (0.75mm);
\fill (250:1.75) circle (0.75mm);
\fill (290:2.25) circle (0.75mm);

\node at (-2,-3.75) {\small $1$};
\node at (-1.5,-3.75) {\small $2$};
\node at (-1,-3.75) {\small $3$};
\node at (-0.5,-3.75) {\small $4$};
\node at (0.625,-3.75) {\small $5$};
\node at (1.125,-3.75) {\small $6$};
\node at (2,-3.75) {\small $8$};
\node at (1.5,-3.75) {\small $7$};

\end{tikzpicture}
\caption{The circle of radius~$\delta^5$ drawn on the dual graph of a stable genus $1$ curve. The white vertex is the circuit. }\label{fig: circle}
\end{figure}

Given a family of tropical curves, which we think of as a tropical curves with edge lengths in a monoid as before, the position of a vertex need not be comparable to any chosen radius $\delta$. In other words, over one fiber of the family, a vertex may lie inside the circle and in another fiber, it may lie outside the circle. Just as not all versal deformations admit contractions, not all families of tropical curves admit well-defined radii $\delta^m$. 

In order that the tropical moduli problem of curves with a circle be well-defined in families, it is necessary to be able to compare the radius of the circle with the distance of a vertex from the minimal genus $1$ subgraph. We may refine the moduli problem of tropical curves by adding an ordering of the non-circuit vertices of the tropicalization to the data in a combinatorial type.  It follows that on a family of tropical curves with the same order type on its vertices, there is a well-defined circle whose contraction leads to an $m$-stable curve.

\begin{guidingprinciple}
The space $\Mbar_{1,n}^{\mathrm{rad}}$ is the moduli space of families of genus $1$ nodal curves together with the data of a total ordering of the vertices of their tropicalizations by distance from the circuit. For each $m<n$, this determines a unique circle whose corresponding contraction yields 	an elliptic $m$-fold curve. The map to $\Mbar_{1,n}$ forgets the ordering, while the map to $\Mbar_{1,n}(m)$ performs the contraction.
\end{guidingprinciple}

\begin{remark}
Ordering \textbf{all} of the vertices is much more information than is strictly necessary for constructing the contraction.  See Section~\ref{vz-comparison} and \cite{KSP-Thesis} for more parsimonious variants.
\end{remark}

An ordering of the non-circuit vertices of a tropical curve can be incorporated into a logarithmic moduli problem, which can in turn be realized as a blowup. 

\subsection{Smyth's moduli spaces}
Fix positive integers $m<n$ and let $C$ be a connected, reduced, proper curve with arithmetic genus $1$. Let $p_1,\ldots, p_n$ be $n$ distinct smooth marked points, and let $\Sigma = p_1 + \cdots + p_n$. 

\begin{definition}
The curve $(C,p_1,\ldots, p_n)$ is \textbf{$m$-stable} if
\begin{enumerate}
\item $C$ has only nodes and elliptic $\ell$-fold points, with $\ell\leq m$, as singularities.
\item If $E\subset C$ is any connected arithmetic genus $1$ subcurve, 
\[
|E\cap \overline{C\setminus E}|+|E\cap \{p_1,\ldots, p_n\}|>m,
\]
\item $H^0(C,\Omega_C^\vee(\Sigma)) = 0$.
\end{enumerate}

\end{definition}

The first condition is standard, and the third condition forces finiteness of the automorphism group. The second condition is required for separability of the moduli problem, as one must discard curves with small numbers of rational tails around the genus $1$ component and replace them with $m$-fold singularities. The main result of~\cite{Smyth} is the following.

\begin{theorem}
There is a proper and irreducible moduli stack $\Mbar_{1,n}(m)$ over $\spec(\mathbf Z[\frac{1}{6}])$, parametrizing $m$-stable $n$-pointed genus $1$ curves.
\end{theorem}

Note that the restriction on the base is due to the presence of unexpected automorphisms of cuspidal curves in characteristics $2$ and $3$. See the discussion in~\cite[Section~2.1]{Smyth}. 

\subsection{Radially aligned logarithmic curves} \label{sec:aligned}

The additional datum necessary to construct a contraction of a logarithmic curve of genus~$1$ to an $m$-stable curve is a \textbf{radial alignment}.  

Let $S$ be a logarithmic scheme whose underlying scheme is the spectrum of an algebraically closed field, and suppose that $\pi : C \to S$ is a logarithmic curve of genus $1$ over $S$.  Let $\plC$ be the tropicalization of $C$.  We write $\ell(e) \in \overline M_S$ for the length of an edge $e$ of $\plC$ (see Section~\ref{sec: log-trop}).  For each vertex $v$ of $\plC$, there is a unique path consisting of edges $e_1, e_2, \ldots, e_k$ from $v$ to the circuit of $\plC$.  We define
\begin{equation*}
\lambda(v) = \sum_{i=1}^k \ell(e_i) .
\end{equation*}
Then $\lambda$ is a piecewise linear function on $\plC$ with integer slopes along the edges and values in $\overline M_S$.  It therefore corresponds to a global section of $\overline M_C$.

\begin{remark}
The section $\lambda$ may be seen as a map from $C$ to the Artin fan $\mathscr A = [ \A^1 / \Gm ]$.  This map sends the circuit of $C$ to the open point of $\mathscr A$ and has contact order~$1$ along every edge and marking.  As such, it can be viewed as an orientation on the edges of the tropicalization $\plC$ of $C$ that are not contained in the cicuit, with all edges oriented away from the circuit.
\end{remark}

\begin{lemma} \label{lem:lambda-canonical}
Let $C$ be a logarithmic curve over $S$ of genus~$1$.  There is an isomorphism of line bundles $\mathscr O_C(\lambda) \simeq \omega_{C/S}(\Sigma)$, where $\omega_{C/S}$ is the relative dualizing sheaf of $C$ over $S$ and $\Sigma$ is the divisor of markings.
\end{lemma}
\begin{proof}
	Let $C_0$ be the open subcurve of $C$ corresponding to the circuit $\plC_0$ of the tropicalization $\plC$ of $C$.  As $\lambda$ takes the value $0$ on $\plC_0$, the line bundle $\mathscr O(\lambda)$ is trivial on $C_0$.  As $\omega_\pi(\Sigma)$ is also trivial on $C_0$ by Proposition~\ref{prop:dualizing-trivial}, we can now show $\mathscr O(-\lambda)$ and $\omega_\pi(\Sigma)$ agree by comparing their degrees on the rational components of $C$ not in the circuit.


If $v$ is not a vertex of the circuit, then $\lambda$ has slope $-1$ on exactly $1$ edge meeting $v$ and has slope $1$ on all remaining edges.  Therefore $\mathscr O(\lambda)$ has degree $-1 + (n-1) = n-2$, where $n$ is the valence of $v$, which coincides with the degree of $\omega_\pi(\Sigma)$.
\end{proof}

Now suppose that $S$ is a logarithmic scheme.  Let $P = \pi_\star \overline M_C$.  The construction of the previous paragraph gives $\lambda_s \in P_s$ for each geometric point $s$ of $S$.  Note that $P_s = \pi_\star \overline M_{C_s}$ by proper base change for \'etale sheaves~\cite[Th\'eor\`eme~5.1~(i)]{SGA4-XII}.  We prove that these $\lambda_s$ are compatible and glue to a canonical global section in $\Gamma(S, \pi_\star \overline M_C) = \Gamma(C, \overline M_C)$.  

To check the compatibility of the $\lambda_s$, we must show they are stable under the generization map 
\begin{equation*}
P_s \to P_t
\end{equation*}
associated to a geometric specialization $t \leadsto s$.  In fact, this is immediate from the fact that $t \leadsto s$ induces an edge contraction $\plC_s \to \plC_t$ compatible with the morphism $\overline M_{S,s} \to \overline M_{S,t}$.

Returning to the case where the underlying scheme of $S$ is the spectrum of an algebraically closed field, we observe that the section $\lambda$ has a basic ordering property:  if $v$ and $w$ are vertices of $\plC_s$ such that the path from $v$ to the circuit passes through $w$ then $\lambda(v) \geq \lambda(w)$ (recall from Section~\ref{sec:pwl} that we think of sharp monoids as partially ordered abelian groups).  However, in general $\lambda(v)$ and $\lambda(w)$ are not comparable when $v$ and $w$ are arbitrary vertices of $\plC_s$.

\begin{definition}
We say that a logarithmic curve over a logarithmic scheme $S$ is \textbf{radially aligned} if $\lambda(v)$ and $\lambda(w)$ are comparable for all geometric points $s$ of $S$ and all vertices $v, w \in \plC_s$.

We write ${\mathfrak M}_{1,n}^{\mathrm{rad}}$ for the category fibered in groupoids over logarithmic schemes whose fiber over $S$ is the groupoid of radially aligned logarithmic curves over $S$ having arithmetic genus~$1$ and~$n$ marked points.
\end{definition}

The imposition of an order between vertices $v$ and $w$ of $\plC$ corresponds to requiring compatibility among the elements $\lambda(v)$ and $\lambda(w)$ of $\overline M_S$.  This effects a logarithmic modification of $S$, as described in Section~\ref{sec:log-blowup} and Example~\ref{ex:log-blowup} in particular.

Note that the notion of radial alignments, as well as variants which follow later in the paper, are distinct from the alignment condition introduced by Holmes in work on the N\'eron models~\cite{Hol14}.  It is related to the notion of aligned logarithmic structure introduced by Abramovich, Cadman, Fantechi, and the third author~\cite{ACFW}.

\begin{proposition} \label{prop:modification}
	$\fM_{1,n}^{\mathrm{rad}}$ is a logarithmic modification of the stack $\mathfrak M_{1,n}^{\log}$ of proper, connected, $n$-marked, genus~$1$, logarithmic curves.
\end{proposition}

\begin{proof}
This is a local assertion on $\mathfrak M_{1,n}^{\log}$.  It is therefore sufficient to show that for a smooth cover $\fM_{1,n}^{\mathrm{log}}$ by $S$, the base change
\begin{equation*}
S \mathop\times_{\mathfrak M_{1,n}^{\log}} \fM_{1,n}^{\mathrm{rad}} \to S
\end{equation*}
is a logarithmic modification.  We can therefore assume that $\overline M_S$ admits a global chart by a monoid $P$, and that, writing $C$ for the family of logarithmic curves over $S$ classified by the map to $\mathfrak M_{1,n}^{\log}$, the tropicalization $\plC$ of $C$ is induced from a tropical curve metrized by $P$.  In other words, $\plC$ is pulled back from $V = \Spec \ZZ[P]$, as is the function $\lambda$.

	Let $\sigma$ be the rational polyhedral cone dual to $P$.  For each vertex $v \in \plC$, the element $\lambda(v) \in P$ corresponds to a linear function on $\sigma$.  Let $\Sigma$ be the fan obtained by subdividing $\sigma$ along the hyperplanes where $\lambda(v) = \lambda(w)$, as $v$ and $w$ range among vertices of $\plC$, and let $W$ be the associated toric variety.  Then $\Sigma \to \sigma$ is the universal morphism of fans such that the linear functions $\lambda(v)$ on $\sigma$ become pairwise comparable on the cones of $\Sigma$.  The base change of $W$ along $S \to V$ is therefore the universal logarithmic scheme mapping to $S$ in which the sections $\lambda(v)$ of $\overline M_S$ become pairwise locally comparable.  Since this is precisely the condition for a family of logarithmic curves to lie in $\mathfrak M_{1,n}^{\mathrm{rad}}$, we may now recognize that
\begin{equation*}
S \mathop\times_{\mathfrak M_{1,n}^{\log}} \fM_{1,n}^{\mathrm{rad}} \simeq S \mathop\times_V W
\end{equation*}
and therefore that it is a logarithmic modification of $S$.
\end{proof}

\begin{corollary} \label{cor:algebraic}
	$\fM_{1,n}^{\mathrm{rad}}$ is representable by a logarithmically smooth algebraic stack.
\end{corollary}
\begin{proof}
	It is a logarithmic modification of (and in particular logarithmically \'etale over) the logarithmically smooth stack $\mathfrak M_{1,n}^{\log}$, so it is certainly logarithmically smooth.
\end{proof}

\subsection{The minimal logarithmic structure}

Suppose that $S$ is a logarithmic scheme whose underlying scheme $\underline S$ is the spectrum of an algebraically closed field, and that we are given a radially aligned logarithmic curve $C$ over $S$, classified by a morphism $\varphi : S \to \fM_{1,n}^{\mathrm{rad}}$.  By virtue of the representability of $\fM_{1,n}^{\mathrm{rad}}$, the logarithmic structure of $\fM_{1,n}^{\mathrm{rad}}$ pulls back to a logarithmic structure $M$ on $S$, equipped with a morphism of logarithmic structures $M \to M_S$.  The objective of this section is to describe $M$ explicitly.

It will help to recognize that $M$ represents a functor on the category $\LogStr(\underline S) / M_S$, which is equivalent to $\Mon / \overline M_S$, where $\Mon$ is the category of sharp, integral, saturated monoids with sharp homomorphisms, where a sharp homomorphism is one in which every invertible element has a unique preimage (for sharp monoids, this is equivalent to a local homomorphism).  The functor in question is 
\begin{equation*}
F(N) = \fM_{1,n}^{\mathrm{rad}}(\underline S, N) \mathop\times_{\fM_{1,n}^{\mathrm{rad}}(\underline S, M_S)} \bigl\{ [C] \bigr\} ,
\end{equation*}
where $N$ lies in $\Mon / \overline M_S$. In other words, $F(N)$ is the set of radially aligned logarithmic curves over the logarithmic scheme $(\underline S, N)$ that pull back via the morphism $S = (\underline S, M_S) \to (\underline S, N)$ to $C$.

Since $\LogStr(\underline S) / M_S$ is equivalent to $\Mon / \overline M_S$, it will be sufficient to describe the characteristic monoid $\overline M$ of $M$.

\begin{proposition} \label{prop:minimal-monoid}
Let $C$ be a radially aligned logarithmic curve over a logarithmic scheme $S$ whose underlying scheme is the spectrum of an algebraically closed field.  Write $\lambda_S$ for the ``distance from the cicuit'' function on the vertices of the tropicalization of $C$.  Let $A$ be the abelian group freely generated by the edges of the dual graph of $C$.  The minimal monoid of $C$ is the sharpening (the quotient by the subgroup of units) of the submonoid of $A$ generated by the smoothing parameters and the differences $\lambda(w) - \lambda(v)$ whenever $\lambda_S(v) \leq \lambda_S(w)$ in $\overline M_S$.
\end{proposition}
\begin{proof}
Let $M_0$ be the minimal logarithmic structure associated to the logarithmic curve $C$ (without taking account of its radial alignment).  The characteristic monoid $\overline M_0$ is well-known to be freely generated by the edges $e$ of the tropicalization $\plC$ of $C$.  Let $\lambda$ denote the ``distance from the circuit'' function valued in $\overline M_0$ and let $\overline M$ be the submonoid of $\overline M_0^{\rm gp}$ generated by $\overline M_0$ and the differences $\lambda(w) - \lambda(v)$ whenever $\lambda_S(w) - \lambda_S(v) \in \overline M_S$.

Now, suppose that $C' \in F(M')$ for some $M'_S \in \LogStr(\underline S) / M_S$.  Then the tropicalization $\plC'$ of $C'$ has edge lengths in $\overline M'_S$.  We write $\lambda'_S$ for the ``distance from the circuit'' function of $\plC'$.  By the unviersal property of $M_0$, we have a unique morphism $M_0 \to M'$ that induces $C$.  We argue that it factors through $M$.

By definition of radial alignment, the vertices of $\plC'$ are totally ordered by $\lambda'$ and this order is compatible with the homomorphism $M'_S \to M_S$.  But $\plC$ and $\plC'$ have the same underlying graph, so the vertices of $\plC'$ have the \emph{same} total order as those of $\plC$, and therefore whenever $\lambda_S(w) - \lambda_S(v) \in \overline M_S$, the difference $\lambda'_S(w) - \lambda'_S(v)$ is in $\overline M'_S$.  This is exactly what is needed to guarantee the required factorization, which is necessarily unique.
\end{proof}

\begin{corollary} \label{cor:loc-free}
	The minimal characteristic monoid of a radially aligned logarithmic curve with tropicalization $\plC$ is freely generated by the lengths of the edges in the circuit and the nonzero differences $\lambda(v) - \lambda(w)$ for $v$ and $w$ among the vertices of $\plC$.
\end{corollary}
\begin{proof}
	The minimal monoid of the logarithmic curve $C$ is freely generated by the smoothing parameters of the nodes.  The quotient described in Proposition~\ref{prop:minimal-monoid} identifies will identify one smoothing parameter with the sum of other smoothing parameters, but the result of such an identification is always locally free.
\end{proof}

Said differently, one may dualize to obtain a tropical description of the minimal radially aligned monoid. Let $\sigma$ be a cone of abstract tropical curves of genus $1$ tropical curves. Let $\widetilde \sigma\to \sigma$ be the subdivision induced by totally ordering the vertices of the dual graph. The minimal base monoid constructed in the proposition can be understood as follows. If $S = \spec(P\to k)$ is a logarithmic enhancement of a closed point, and $\pi: C\to S$ be a radially aligned logarithmic curve, then there is a canonical morphism of rational polyhedral cones, $P^\vee \to \sigma$.  As $C$ is radially aligned, this morphism factors through some cone in the subdivision $\tilde\sigma$.  There is a minimal such cone with respect to face inclusions, and the minimal monoid is the dual monoid of that cone.  See Figure~\ref{fig: minimal-monoid}.

\begin{corollary} \label{cor:smooth}
	The underlying algebraic stack of $\mathfrak M_{1,n}^{\mathrm{rad}}$ is smooth.
\end{corollary}
\begin{proof}
	We saw in Corollary~\ref{cor:algebraic} that it is logarithmically smooth and in Corollary~\ref{cor:loc-free} that its logarithmic structure is locally free.
\end{proof}

\begin{figure}
\includegraphics[scale=0.35]{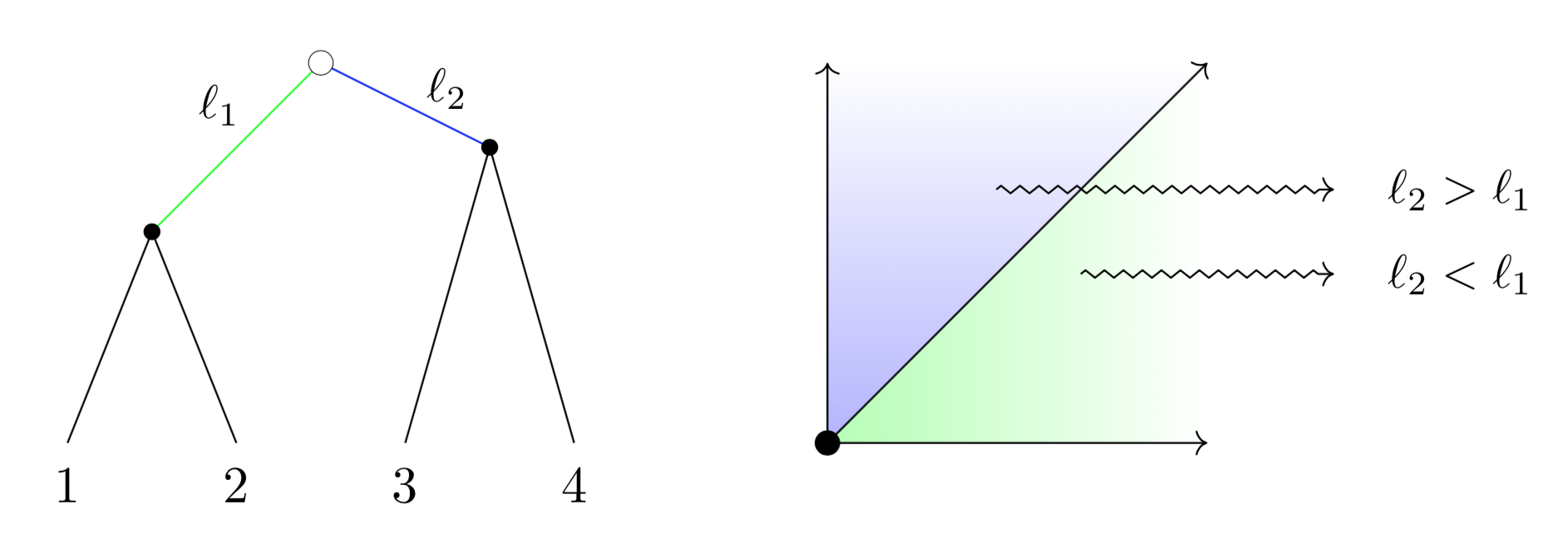}
%
%
%
%
%
%
%
%
\caption{The cone on the right without its subdivision is the minimal monoid of a logarithmic curve with dual graph on the left. Each of the cones of a subdivision is a different minimal radially aligned curve.}\label{fig: minimal-monoid}
\end{figure}

\subsection{Circles around the circuit}
\label{sec:circles-circuit}

We introduce a logarithmic version of Smyth's $m$-stability conditions~\cite[Section 1]{Smyth}.

\begin{definition} \label{def:valence}
Let $C$ be a radially aligned logarithmic curve over a logarithmic scheme $S$ whose underlying scheme is the spectrum of an algebraically closed field.  Let $\plC$ be the tropicalization of $C$.  Let $\lambda$ be the ``distance from the circuit'' function on the vertices of $\plC$.  Suppose that $\delta \in \overline M_S$.  We say that $\delta$ is \textbf{comparable to the radii} of $C$ if it is comparable to $\lambda(v)$ for all vertices $v$ of $\plC$.  

Let $e$ be an edge of $\plC$ incident to vertices $v$ and $w$ with $\lambda(v) < \lambda(w)$.  We say that $e$ is \textbf{incident} to the circle of radius $\delta$ if $\lambda(v) < \delta \leq \lambda(w)$.  We say that $e$ is \textbf{excident} to the circle of radius $\delta$ around the circuit of $\plC$ if $\lambda(v) \leq \delta < \lambda(w)$.

We define the \textbf{inner valence} and \textbf{outer valence} of $\delta$, respectively, to be the number of edges of $\plC$ incident and excident from the circle of radius $\delta$.
\end{definition}

Some remarks about this definition are in order:
\begin{enumerate}[(A)]
\item Intuitively, an edge of $\plC$ is incident to the circle of radius $\delta$ if it crosses the circle.  This concept becomes ambiguous when the circle crosses a vertex of $\plC$, where we must distinguish edges that contact the circle from the inside from those that contact it from the outside.
\item If an edge $e$ of $\plC$ connects vertices $v$ and $w$ that are not both on the circuit then either $\lambda(v) < \lambda(w)$ or $\lambda(w) < \lambda(v)$.  By definition of radial alignment, we have one or the other non-strict inequality.  But equality is impossible, for $\lambda(v) - \lambda(w) = \pm \delta(e)$, where $\delta(e)$ is the smoothing parameter of $e$ and in particular is nonzero.  There is no way for the edge to lie \emph{within} the circle of radius $\delta$.
\item If $v$ is a vertex of the tropicalization $\plC$ of a \emph{stable}, radially aligned logarithmic curve and $v$ is not on the circuit then there is exactly one edge of $\plC$ incident to $v$ and at least two edges (including legs) of $\plC$ excident from $v$.  If the curve is merely semistable then there is still one incident edge and at least one excident edge.  We leave the verification of these statements to the reader.
\item It follows from the previous observation that the inner valence of the circle of radius $\delta$ on a \emph{semistable}, radially aligned logarithmic curve is always bounded above by the outer valence.
\end{enumerate}

\begin{proposition}
Suppose that $C$ is a radially aligned, semistable logarithmic curve over $S$ and that $\delta$ is a global section of $\overline M_S$ that is comparable to the radii of $C$.  For each geometric point $s$ of $S$, let $\eta(s)$ and $\tau(s)$ be the inner and outer valence, respectively, of the circle of radius $\delta$ on the tropicalization of $C$.  Then $\eta$ is upper semicontinuous and $\tau$ is lower semicontinuous.
\end{proposition}
\begin{proof}
As $\eta$ and $\tau$ are constant on the logarithmic strata of $S$, they are constructible functions.  It is therefore sufficient to show that for every geometric specialization $t \leadsto s$ of $S$, we have $\eta(t) \leq \eta(s)$ and $\tau(t) \geq \tau(s)$.  But if $\plC_s$ and $\plC_t$ denote the tropicalizations of $C_s$ and $C_t$ then $\plC_t$ is obtained from $\plC_s$ by a weighted edge contraction.  The proposition follows from the following three observations:
\begin{enumerate}[(1)]
\item Contracting edges that are neither incident to $\delta$ nor excident from it does not change $\eta$ or $\tau$.
\item Contracting edges incident to $\delta$ does not change $\tau$ but may decrease $\eta$.
\item Contracting edges excident from $\delta$ does not change $\eta$ but may increase $\tau$.
\end{enumerate}
\end{proof}

\begin{definition} \label{def:delta_m}
Let $C$ be a family of radially aligned genus~$1$ logarithmic curves over $S$.  For each integer $m$ such that $0 \leq m \leq n$, we say that $\delta \in \overline M_S$ is \textbf{$m$-stable} if
\begin{enumerate}[(i)]
\item $\delta_s$ is comparable to $\lambda_s(v)$ for all vertices $v$ of $\plC_s$, and 
\item the circle of radius $\delta_s$ around the circuit of $\plC_s$ has inner valence $\leq m$ and outer valence $> m$.
\end{enumerate}
	If an $m$-stable radius exists, we write $\delta^m$ for the \emph{smallest} $m$-stable radius.
\end{definition}

\begin{proposition} \label{prop:m-stable}
	If $C$ is a \emph{semistable}, radially aligned logarithmic curve over $S$, where the underlying scheme of $S$ is the spectrum of an algebraically closed field, then an $m$-stable radius exists.
\end{proposition}
\begin{proof}
	Let $\Lambda$ be the set of $\lambda(v)$, as $v$ ranges among the vertices of the tropicalization $\plC$ of $C$.  Then $\Lambda$ inherits a total order from $\overline M_S$.  Every vertex of $\plC$ not in the circuit is the endpoint of exactly one incident vertex:  it is at least one because the graph is connected and at most one because any more would increase the genus beyond~$1$.  On the other hand, because the graph is semistable, every vertex outside the circuit is an endpoint of at least~$2$ edges, and therefore has at least~$1$ excident edge.  Both the number of incident edges, and the number excident edges, to the circle of radius $\delta$ are therefore increasing functions of $\delta \in \Lambda$.  For the maximal $\delta \in \Lambda$, the external valence is $n$, and for $\delta = 0$ the internal valence is $0$, so there must be some $m$-stable $\delta$ in between. 
\end{proof}

\subsection{The universal curves} 
\label{sec:univeral-curves}

Let $C$ be a radially aligned, semistable logarithmic curve over $S$ and let $\delta$ be a section of $\overline M_S$ that is comparable to the radii of $C$ (Definition~\ref{def:valence}).

\begin{proposition} \label{prop:C_delta}
There is a universal logarithmic modification $C_\delta \to C$ such that the sections $\lambda$ and $\delta$ of $\overline M_{C_\delta}$ are comparable.  The corresponding map on tropicalizations $\plC_\delta \to \plC$ subdivides the edges that are simultaneously incident to and excident from the circle of radius $\delta$ along the circle.
\end{proposition}
\begin{proof}
Let $\plC$ be the tropicalization of $C$. The section $\delta-\lambda$ gives a map $\plC\to \RR$ in the obvious fashion. Subdivide $\plC$ along the preimage of $0\in \RR$. This subdivision of $\plC$ gives rise to a logarithmic modification $C_\delta$ of $C$. The conclusion about tropicalizations is true by construction.
\end{proof}


Apply the proposition with the values $\delta^m$ introduced at the end of Section~\ref{sec:circles-circuit}, to construct curves $C = \widetilde C_0, \ldots, \widetilde C_n$ over $\fM^{\mathrm{rad}}_{1,n}$, each of which is equipped with a stabilization $\widetilde C_i \to C$.


%

\subsection{Resolution of indeterminacy}

We define $\Mbar_{1,n}^{\mathrm{rad}}$ to make the following square cartesian:
\[
\begin{tikzcd}
\Mbar_{1,n}^{\mathrm{rad}}\arrow{d}\arrow{r} & \Mbar_{1,n} \arrow{d} \\
{\mathfrak M}^{\mathrm{rad}}_{1,n} \arrow{r} & \mathfrak M_{1,n}^{\log}
\end{tikzcd}
\]
As we have already seen, the bottom arrow is a logarithmic modification. As the pullback of a logarithmic modification is a logarithmic modification, $\Mbar_{1,n}^{\mathrm{rad}}$ is a logarithmic modification of $\Mbar_{1,n}$.

For each $m$, we construct a projection from $\Mbar_{1,n}^{\mathrm{rad}}$ to Smyth's moduli spaces $\overline{\mathcal M}_{1,n}(m)$ of $m$-prestable curves, resolving the indeterminacy of the map $\overline{\mathcal M}_{1,n} \dashrightarrow \overline{\mathcal M}_{1,n}(m)$.

\begin{theorem}\label{thm: contraction to smyth}
For each integer $m$ such that $0 \leq m \leq n$, there is a proper, birational morphism $\phi_m: \Mbar_{1,n}^{\mathrm{rad}} \to \overline{\mathcal M}_{1,n}(m)$.
\end{theorem}

The main point of the proof is the construction of a contraction $\widetilde C_m \to \overline C_m$ where $\widetilde C_m$ is the curve defined in Section~\ref{sec:univeral-curves} and $\overline C_m$ is a Smyth $m$-stable curve.  The construction uses the section $\delta^m$ to produce a line bundle on $\widetilde C_m$ and then recognizes $\overline C_m$ as $\Proj$ of the section ring of this bundle. \medskip

\noindent
\textit{{\bf Notation:} We will hold $m$ fixed for the rest of this section, so we drop the subscript in what follows.}\medskip

\begin{definition} \label{def:mu}
Let $C$ be a radially aligned logarithmic curve over $S$ and let $\delta$ be a section of $\overline M_S$ that is comparable to the radii of $C$ (Definition~\ref{def:valence}).  Then, by construction of $C_\delta$ (Proposition~\ref{prop:C_delta}), $\lambda$ and $\delta$ are comparable sections of $\overline M_{C_\delta}$.  Therefore, there is a well-defined section $\mu = \max \{ \lambda, \delta \}$ on $C_\delta$.
\end{definition}

\begin{lemma} \label{lem:L-degree}
Assume that $C$ is a semistable logarithmic curve over $S$.  The degree of $\mathcal O_{\widetilde C}(\mu)$ is nonnegative on all components of all geometric fibers of $\widetilde C$ over $S$.  For all geometric points $s$ of $S$ and all components $\widetilde C_v$ of $\widetilde C_s$ such that $\lambda_s(v) < \delta_s$, the degree of $\mathcal O_{\widetilde C}(\mu)$ on $\widetilde C_v$ is zero.  If $v$ is not in the interior of the circle of radius $\delta_s$ then $L$ has positive degree on $\widetilde C_v$.
\end{lemma}
\begin{proof}
It is sufficient to consider the case where the underlying scheme of $S$ is the spectrum of an algebraically closed field.  Let $\widetilde \plC$ be the tropicalization of $\widetilde C$.  If $v$ is in the interior of the circle of radius $\delta$ on $\widetilde \plC$ then by definition $\lambda(v) < \delta$ so $\mu(v) = \delta$.  Therefore the restriction of $L$ to $C_v$ is pulled back from $S$ and in particular has degree~$0$.  

If $v$ is in the exterior of the circle of radius $\delta$ then $\mu$ agrees with $\lambda$ at $v$ and we know from lemma~\ref{lem:lambda-canonical} that $\mathscr O_C(\lambda)$ has positive degree on $v$. Finally, if $v$ is on the boundary of the circle of radius $\delta$ then $v$ has exactly one incident edge and at least one excident edge.  But $\mu$ is constant on the incident edge, so the degree of $\mathscr O_C(\mu)$ is at least~$1$.
\end{proof}

\setcounter{subsubsection}{\value{theorem}}
\numberwithin{theorem}{subsubsection}
\numberwithin{equation}{subsubsection}
\subsubsection{The circuit}

For this section, assume that $C$ is a family of radially aligned semistable logarithmic curves over $S$, that $\delta$ is a section of $\overline M_S$ that is comparable to the radii of $C$, and that $\lambda$ and $\delta$ are comparable on $C$.  Let $\pi : C \to S$ be the projection.

Recall that we have defined $\mu$ to be the section $\max \{ \lambda, \delta \}$ on $C$.  Since $\lambda \leq \mu$, we have a morphism of invertible sheaves (see Section~\ref{sec:pwl} for the construction):
\begin{equation} \label{eqn:E}
i : \mathscr O_C(\lambda) \to \mathscr O_C(\mu)
\end{equation}

\setcounter{theorem}{\value{equation}}
\begin{definition} \label{def:E}
We write $E_\delta$ for the support of the cokernel and call it the \emph{circuit (of radius $\delta$)} in $C$.  Note that $E_\delta$ represents the subfunctor of $C$ where $\lambda < \delta$. We will suppress the subscript when it is clear from context.
\end{definition}

\begin{lemma}
Suppose that $\mathscr O_S \to \mathscr O_S(\delta)$ is injective.  Then $\mathscr O_C(\lambda) \to \mathscr O_C(\mu)$ is injective and $E_\delta$ is a Cartier divisor on $C$.
\end{lemma}
\begin{proof}
Since $\lambda \leq \mu \leq \lambda + \delta$ we have a sequence of maps
\begin{equation*}
\mathscr O_C(\lambda) \xrightarrow{i} \mathscr O_C(\mu) \to \mathscr O_C(\lambda + \delta)
\end{equation*}
where the composition is a twist of the pullback of the injection
\begin{equation*}
\mathscr O_S \to \mathscr O_S(\delta) 
\end{equation*}
by $\lambda$. As $C$ is flat over $S$, this implies that $\mathscr O_C(\lambda) \to \mathscr O_C(\lambda + \delta)$ and, a fortiori, $\mathscr O_C(\lambda) \to \mathscr O_C(\mu)$ are injective.
\end{proof}


\begin{definition} \label{def:Delta}
	Let $\Delta_\delta$ (or $\Delta$, when the dependence on $\delta$ is evident) be the locus in $S$ where the map $\mathscr O_S(-\delta) \to \mathscr O_S$ vanishes.
\end{definition}

\begin{lemma} \label{lem:cohom-kmu}
Assume that $\mathscr O_S \to \mathscr O_S(\delta)$ is injective and that each fiber of $C$ contains at least one component not in the interior of the circle of radius $\delta$.
For all integers $k > 0$, we have
$R^1 \pi_\star \mathscr O_C(k \mu) = R^1 \pi_\star \mathscr O_E(k \delta) = \mathbf E^\vee_\Delta(k \delta)$ where $\mathbf E^\vee_\Delta$ is the restriction of the dual of the Hodge bundle of $C$ over $S$ to $\Delta$.
\end{lemma}
\numberwithin{paragraph}{theorem}
\noindent\textit{Proof}.
\paragraph{}
Recalling that, by definition, $E$ is the locus where $\mathscr O_C(\lambda) \xrightarrow{i} \mathscr O_C(\mu)$ vanishes, we have an exact sequence:
\begin{equation*}
	0 \to \mathscr O_C(k \lambda) \xrightarrow{i^k} \mathscr O_C(k \mu) \to \mathscr O_E(k \mu) \to 0 
\end{equation*}
	Note that $i$ is injective because $\mu - \lambda \leq \delta$ and $\mathscr O_S \to \mathscr O_S(\delta)$ is injective.

As $\mathscr O_C(\lambda) = \omega_{C/S}(\Sigma)$ by Lemma~\ref{lem:lambda-canonical}, and as $\mu$ coincides with $\delta$ on $E$, this simplifies:
\begin{equation*}
0 \to \omega_{C/S}^{\otimes k}(k\Sigma) \to \mathscr O_C(k \mu) \to \mathscr O_E(k \delta) \to 0
\end{equation*}
	We have $R^2 \pi_\star \omega_{C/S}^{\otimes k}(k\Sigma) = 0$ because the fibers are $1$-dimensional and $R^1 \pi_\star \omega_{C/S}^{\otimes k}(k \Sigma) = 0$ because its fibers vanish by Corollary~\ref{cor:twisted-dualizing}.  The isomorphism $R^1 \pi_\star \mathscr O_C(k\mu) \simeq R^1 \pi_\star \mathscr O_E(k\delta)$ follows.

\paragraph{\textit{Filtering $E$ into flat pieces.}}
	To conclude the lemma, it remains to identify each of these with a twist of the dual of the Hodge bundle. We argue that the map
\setcounter{equation}{\value{paragraph}}
\numberwithin{equation}{theorem}
\begin{equation} \label{eqn:1}
R^1 \pi_\star \mathscr O_{C_\delta} \to R^1 \pi_\star \mathscr O_E
\end{equation}
\setcounter{paragraph}{\value{equation}}%
is an isomorphism, where $C_\delta = \pi^{-1} \Delta$.  This assertion is local in $S$, so we can assume that $S$ is an atomic neighborhood of a geometric point $s$ of $S$.  Since the tropicalization of $C_s$ is radially aligned, there is a sequence of radii given by
\begin{equation*}
0 = \delta_0 < \delta_1 < \cdots < \delta_{n} < \delta_{n+1} = \delta
\end{equation*}
given by the distance of the vertices from the elliptic component, terminating at $\delta$. For each $i$, define $\mu_i = \max \{ \lambda, \delta_i \}$ and define $E_i$ by the ideal $\mathscr O_C(\lambda - \mu_i) \subset \mathscr O_C$.  We take $\Delta_{\delta_i}$ to be the support of the cokernel of $\mathscr O_S \to \mathscr O_S(\delta_i)$. The pieces $E_i$ filter $E$ into pieces that are flat\footnote{In the generic case, where there is a unique genus $1$ component inside the circle $E$ is flat over $\Delta$ and the result follows from Serre duality.} over their images in $S$, noting that $E$ is not necessarily flat over its image, $\Delta$.

	We now prove~\eqref{eqn:1} by induction, with $E$ replaced by $E_i$ and $\Delta$ replaced by $\Delta_{\delta_i}$ for all $i$.  When $i = 0$, we have $E_i = \varnothing$ and $\Delta_{\delta_i} = \varnothing$, and the assertion is trivial.  The kernel of $\mathscr O_{E_{i+1}} \to \mathscr O_E$ is isomorphic to the cokernel of the canonical map $\mathscr O_C(\lambda - \mu_{i+1}) \to \mathscr O_C(\lambda - \mu_i)$.  Let $D$ be the support of this cokernel.  

\paragraph{\textit{Flatness of the pieces}}
	We claim that $D$ is flat over the locus $\Delta_{\delta_{i+1} - \delta_i} \subset S$. This can be seen explicitly as follows. By construction, $D$ is defined by the ideal $\mathscr O_C(\mu_i - \mu_{i+1}) \subset \mathscr O_C$. The central fiber of $D_s$ thus consists of those components of $C_s$ where $\lambda < \delta_{i+1}$. Furthermore, $\mu_i - \mu_{i+1}$ takes the constant value $\delta_i - \delta_{i+1}$ except at the nodes where $D_s$ is joined to the rest of $C_s$.  This implies that, away from those nodes, $D_s$ is defined by the preimage in $\mathscr O_C$ of the ideal $\mathscr O_S(\delta_i - \delta_{i+1}) \subset \mathscr O_S$ that defines $\Delta_{\delta_{i+1} - \delta_i}$. Thus $D$ is flat over the claimed locus, except possibly at the nodes where $D_s$ is joined to its complement in $C_s$. At such a node $p$ of $C_s$ with local equation $xy = t$, calculation shows that $\mathscr O_{C,p}$ is \'etale-locally isomorphic to $\mathscr O_{S,s}[x,y] / (xy - t, x) \simeq \mathscr O_{S,s}[y] / (t)$, which is flat over $\mathscr O_{S,s} / (t)$, as claimed.
    


\paragraph{\textit{Only the central component contributes cohomology.}}
	We introduce the notation $C_\gamma = \pi^{-1} \Delta_\gamma$ and consider the following commutative diagram with exact rows:
\begin{equation*} \xymatrix{
0 \ar[r] & \mathscr O_{C_{\delta_{i+1}-\delta_i}}(-\delta_i) \ar[r] \ar[d] & \mathscr O_{C_{\delta_{i+1}}} \ar[r] \ar[d] & \mathscr O_{C_{\delta_i}} \ar[r] \ar[d] & 0 \\
0 \ar[r] & \mathscr O_D(\lambda - \mu_i) \ar[r] & \mathscr O_{E_{i+1}} \ar[r] & \mathscr O_{E_i} \ar[r] & 0
} \end{equation*}
Pushing forward along $\pi$, we obtain the following diagram, again with exact rows:

\noindent\resizebox{\textwidth}{!}{$\displaystyle%
\xymatrix{
0 \ar[r] & \pi_\star \mathscr O_{C_{\delta_{i+1}-\delta_i}}(-\delta_i) \ar[r] \ar[d]^\alpha & \pi_\star \mathscr O_{C_{\delta_{i+1}}} \ar[r] \ar[d]^{\sigma_{i+1}} & \pi_\star \mathscr O_{C_{\delta_i}} \ar[r] \ar[d]^{\sigma_i} & R^1 \pi_\star \mathscr O_{C_{\delta_{i+1}-\delta_i}}(-\delta_i) \ar[r] \ar[d]^\beta & R^1 \pi_\star \mathscr O_{C_{\delta_{i+1}}} \ar[r] \ar[d]^{\varphi_{i+1}} & R^1 \pi_\star\mathscr O_{C_{\delta_i}} \ar[r] \ar[d]^{\varphi_i} & 0 \\
0 \ar[r] & \pi_\star \mathscr O_D(\lambda - \mu_i) \ar[r] & \pi_\star \mathscr O_{E_{i+1}} \ar[r] & \pi_\star \mathscr O_{E_i} \ar[r] & R^1 \pi_\star \mathscr O_D(\lambda - \mu_i) \ar[r] & R^1 \pi_\star \mathscr O_{E_{i+1}} \ar[r] & R^1 \pi_\star \mathscr O_{E_i} \ar[r] & 0
}$}

By induction on $i$, we show that $\sigma_i$ and $\varphi_i$ are isomorphisms by demonstrating that both $\alpha$ and $\beta$ are isomorphisms.

Our task is now to show that
\begin{equation}\label{eqn:2}
\mathrm R\pi_\star\mathscr O_{C_{\delta_{i+1}-\delta_i}}(-\delta_i) \to \mathrm R\pi_\star\mathscr O_{D}(\lambda - \mu_i)
\end{equation}
is a quasi-isomorphism.  Since both $\mathscr O_{C_{\delta_{i+1}-\delta_i}}(-\delta_i)$ and $\mathscr O_D(\lambda - \mu_i)$ are flat over $\Delta_{\delta_{i+1}-\delta_i}$, both $\mathrm R \pi_\star\mathscr O_{C_{\delta_{i+1}-\delta_i}}$ and $\mathrm R \pi_\star\mathscr O_D(\lambda - \mu_i)$ are representable by boudned complexes of locally free sheaves.  Nakayama's lemma shows that it is sufficient to verify this on the fibers. 

We are therefore to show that
\begin{equation} \label{eqn:3}
\mathrm R\pi_\star\mathscr O_{C_s}(-\delta_i) \to \mathrm R\pi_\star\mathscr O_{D_s}(\lambda-\mu_i)
\end{equation}
is an isomorphism, where $s$ is a geometric point of $\Delta_{\delta_{i+1}-\delta_i}$.  Let $E_0\subset D_s$ be the minimal closed genus~$1$ subcurve of $C_s$.  We claim that there are quasi-isomorphisms
\begin{align*}
\mathrm R\pi_\star\mathscr O_{C_s}(-\delta_i) & \to \mathrm R\pi_\star\mathscr O_{E_0}(-\delta_i)  &&\text{and}  \\
\mathrm R\pi_\star\mathscr O_{E_0}(-\delta_i) & \to  \mathrm R\pi_\star\mathscr O_{D_s}(\lambda-\mu_i)  
\end{align*}
commuting with~\eqref{eqn:3}. 

The first map is induced by the quotient $\mathscr O_{C_s} \to \mathscr O_{E_0}$. This induces an isomorphism on cohomology since $C_s$ and $E_0 \subset C_s$ are reduced, proper, connected curves of genus~$1$.

For the second map, we induct on $i$.  For each $j \in \{ 0, \ldots, i\}$, let $D_j$ be the union of components of $D_s$ where $\lambda \leq \delta_{j+1}$.  Recall that $D_{i} = D_s$ because $\lambda \leq \delta_{i+1}$ on $D_s$. We have $E_0 = D_0$ because $\delta_0 = 0$.  For each $j$, there is an exact sequence
\begin{equation*}
0 \to \mathscr O_{D_{j-1}}(\lambda-\mu_{j-1}-\delta_j+\delta_{j-1}) \to \mathscr O_{D_j}(\lambda - \mu_j) \to \mathscr O_{F_j}(\lambda - \mu_j) \to 0
\end{equation*}
where $F_j$ is the closure of the complement of $D_{j-1}$ in $D_j$.  Note that $F_j$ is a disjoint union of smooth rational curves.  Let $v$ be the vertex of the dual graph of $C_s$ corresponding to a component of $F_j$.  Since $v$ is on the boundary of the circle of radius $\delta_j$, the piecewise linear function $\lambda - \mu_j$ has outgoing slope $0$ along all but one of the edges incident to $v$.  The remaining edge connects $v$ to the interior of the circle of radius $\delta_j$ and therefore the outgoing slope of $\lambda - \mu_j$ is $-1$ along that edge.  It follows that, on each component of $F_j$, the restriction of $\mathscr O_{F_j}(\lambda-\mu_j)$ is $\mathscr O_{\mathbf P^1}(-1)$.  Thus $\mathrm R\pi_\star \mathscr O_{F_j}(\lambda-\mu_j) = 0$. It follows that $\mathrm R \pi_\star \mathscr O_{D_{j-1}}(\lambda-\mu_{i-1}-\delta_j+\delta_{j-1}) \simeq \mathrm R \pi_\star \mathscr O_{D_j}(\lambda-\mu_j)$.  The desired quasi-isomorphism is constructed by induction. 
\qed

\numberwithin{equation}{subsubsection}
\subsubsection{Flatness of the section ring}

We continue to assume that $C$ is a radially aligned logarithmic curve over $S$, that $\delta$ is a section of $\overline M_S$ comparable to the radii of $C$, and that $\mathscr O_S \to \mathscr O_S(\delta)$ is injective.

Under these assumptions, Lemma~\ref{lem:cohom-kmu} supplies a canonical resolution of $R^1 \pi_\star \mathscr O_C(k \mu)$:
\begin{equation} \label{eqn:resolution}
0 \to \mathbf E^\vee\bigl((k-1)\delta\bigr) \to \mathbf E^\vee(k\delta) \to \R^1 \pi_\star \mathscr O_C(k \mu) \to 0
\end{equation},
where $\mathbf E$ denotes the Hodge bundle.
Note that the injectivity on the left comes from the injectivity of $\mathscr O_S \to \mathscr O_S(\delta)$. We come to the key proposition necessary to contract radially aligned curves in families. 

\setcounter{theorem}{\value{equation}}
\begin{proposition}\label{locallyfree}
The sheaf $\pi_{\star} \mathscr O_C(k\mu)$ is locally free for all $k \geq 0$.
\end{proposition}

Away from $\Delta$, we can identify $\mathscr O_C(k\mu) \simeq \omega_{C/S}^{\otimes k}(k\Sigma)$, and we know that $\pi_{\star} \bigl(\omega_{C/S}^{\otimes k}(k\Sigma) \bigr)$ is locally free of the expected rank for all $k \geq 0$.  It therefore suffices to work near $\Delta$. The following lemma will allow us to reduce the proof of Proposition~\ref{locallyfree} to the case where $S$ is the spectrum of a discrete valuation ring.

\numberwithin{equation}{theorem}
\begin{lemma}\label{boundary}
	Let $T \to S$ be a morphism such that $\mathscr O_T(-\delta) \to \mathscr O_T$ is injective.  Then     
\[
f^{\star}\pi_{\star} \mathscr O_C(k \mu) =\pi_{\star}f^{\star} \mathscr O_C(k\mu).
\]
\end{lemma}
\begin{proof}
We write $L = \mathscr O_C(k\mu)$.

	Working locally near the image of $T$, the proof of cohomology and base change \cite[second Theorem, \S5, p.\ 46]{Mumford-AV} guarantees we can find $K^0, K^1$ finitely generated and locally free fitting into an exact sequence~\eqref{exact}
\begin{equation}\label{exact}
0\to\pi_\star L\to K^0\to K^1 \to \R^1\pi_\star L\to 0
\end{equation}
	such that~\eqref{eqn:exact2} is exact as well.
\begin{equation}\label{eqn:exact2}
0\to \pi_\star f^\star L\to f^\star K^0 \to f^\star K^1 \to \R^1\pi_\star f^\star L\to 0
\end{equation}

We show that the sequence
\[
0\to f^\star \pi_\star L \to f^\star K^0 \to f^\star K^1\to f^\star \R^1\pi_\star L \to 0
\]
is exact, from which it follows that $\pi_\star f^\star L \simeq f^\star \pi_\star L$ via the natural map.

We perform a derived pullback on the sequence~\eqref{exact} along $f$, yielding a spectral sequence $\L_pf^{\star}\R^q\pi_{\star}L$ converging to 0. A diagram chase shows that the obstructions to the desired isomorphism come from $\L_1f^{\star}\R^1\pi_{\star}L$ and $\L_2f^{\star}\R^1\pi_{\star}L$. We will use our explicit resolution of $\R^1\pi_{\star}L$ in equation~\eqref{eqn:resolution} to show that both of these groups vanish.  

Working locally, we rewrite the resolution~\eqref{eqn:resolution} as 
\[
	0 \to \mathscr O_S(-\delta) \xrightarrow{c} \mathscr O_S\to \R^1 \pi_\star L \to 0
\]
for some local section $c$ of $\mathscr O_S$ that vanishes along $\Delta$.  The pullback of this sequence to $T$ is exact by assumption.  Therefore $\L_1f^\star\R^1\pi_\star L$ and $\L_2f^\star\R^1 \pi_\star L$ both vanish, as required, and $f^\star\pi_\star L=\pi_\star f^\star L$.
\end{proof}

\begin{proof}[Proof of Proposition~\ref{locallyfree}]
	By Lemma~\ref{boundary}, it is sufficient to treat the universal case, where $S = \frak M_{1,n}^{\mathrm{rad}}$.  Since $\frak M_{1,n}^{\mathrm{rad}}$ is reduced, it suffices to prove that $\pi_\star \mathscr O_C(k\mu)$ has constant rank.  Since $\mathscr O_S \to \mathscr O_S(\delta)$ is injective, every point of $S$ has a generization where it restricts to an isomorphism (i.e., where $\delta = 0$).  If $s$ is a point of $S$, we can therefore find a scheme $T$, the spectrum of a discrete valuation ring, and a map $f : T \to S$ carrying the closed point to $s$ and the generic point to the complement of $\Delta$.  By Lemma~\ref{boundary}, the formation of $\pi_\star \mathscr O_C(k\mu)$ commutes with base change to $T$, so we can replace $S$ with $T$.

Now, $\mathscr O_C(k\mu)$ is torsion free, so $\pi_\star \mathscr O_C(k\mu)$ is also torsion free, hence flat because $S$ is the spectrum of a discrete valuation ring.
\end{proof}

\subsubsection{Contraction to $m$-stable curves}
We are now prepared to complete our contraction of radially aligned curves to the $m$-stable curves. Our argument is in the spirit of Smyth's contraction lemma \cite[Lemma 2.12]{Smyth}. The major difference in the present setting is that the extra datum of the circle of fixed radius allows us to promote Smyth's local construction to a global one.

Let $\pi : C \to S$ be a radially aligned, semistable, genus~$1$ logarithmic curve over $S$ and let $\delta$ be a section of $\overline M_S$ that is comparable to the radii of $C$.  Assume that $\mathscr O_S \to \mathscr O_S(\delta)$ is injective.  We collect from our earlier discussion
\begin{enumerate}
\item a section $\mu = \max \{ \lambda, \delta \}$ of $\overline M_C$ (Definition~\ref{def:mu});
\item a line bundle $\mathscr O_C(\mu)$ on $C$ (Definition~\ref{def:mu});
\item a Cartier divisor $E$ on $C$ (Definition~\ref{def:E}), the locus in $C$ where $\lambda < \delta$;
\item a divisor $\Delta$ on $S$ (Definition~\ref{def:Delta}), the locus in $S$ where $\delta > 0$; and
\item that $\pi_{\star} \mathscr O_C(k\mu)$ is locally free for all $k \geq 0$ (Lemma~\ref{boundary}).
\end{enumerate}


\begin{proposition}\label{contraction}
Given the above situation, $\mathscr O_C(\mu)$ is $\pi$-semiample and we have a diagram
\[
\begin{tikzcd}
	C \arrow[dr,"\pi"'] \arrow{rr}{\tau} & & \overline C:=\smash{\mathrm{Proj} \bigl( \sum_{k\geq 0} \pi_{\star} \mathscr O_C(k\mu)\bigr) \arrow[start anchor={south west}]{dl}{\overline\pi}}\\
&S &
\end{tikzcd}
\]
with $\tau$ proper, birational, with exceptional locus $E$. Furthermore, 
\begin{enumerate}
\item $\overline\pi:\overline C\to S$ is flat and projective with reduced fibers;
\item $\tau|_{\overline{{C_s}\setminus E_s}}:\overline{{C_s}\setminus E_s}\to \overline C_s$ is the normalization of $\overline C_{s}$ at $\tau(E_{s})$ for each fiber over each geometric point $s$ of $S$;
\item $\tau(E_{s})$ is an elliptic $m$-fold point in each $\overline C_{s}$ over each geometric point $s$ of $S$, and $\overline C\to S$ together with the image of $\Sigma$ is an $m$-stable curve in the sense of Smyth.
\end{enumerate}
\end{proposition}

\begin{proof}
We know that $\pi_{\star} \mathscr O_C(k\mu)$ is locally free for all $k \geq 0$ by Proposition \ref{locallyfree}, so $\overline C\to S$ is flat. 

Observe that $\mathscr O_C(\mu)$ being $\pi$-semiample is equivalent to the surjectivity of the adjunction map
\[
\pi^\star \pi_{\star} \mathscr O_C(k\mu) \to \mathscr O_C(k\mu)
\]
for $k$ sufficiently large. Note that we have that $\mathscr O_C(k\mu)$ is ample on generic fibers, and over $\Delta$,
\[
\mathscr O_E(k\mu)\simeq \mathscr{O}_E \ \  \text{and} \ \ \mathscr O_{C_s \setminus E}(k\mu) \textnormal{ is ample. }
\]
We must argue that, for every $x \in C$, there is some $k \geq 0$ and a section of $\mathscr O_C(k\mu)$ that does not vanish at $x$, at least in a neighborhood of $\pi(x)$ on $S$.  Since $\mathscr O_C(\mu)$ coincides with $\omega_{C/S}(\Sigma)$ over $S \setminus \Delta$, and $\omega_{C/S}(\Sigma)$ is semiample on $C$, this presents no obstacle away from $\Delta$.  Even over $\Delta$, the restriction of of $\mathscr O_C(\mu)$ to the complement of $E$ agrees with $\omega_{C/S}(\Sigma)$ on components that do not meet $E$, and with $\omega_{C/S}(\Sigma-p)$ on a component attached at $p$ to $E$.  Since $C$ is semistable, $\omega_{C/S}(\Sigma-p)$ has degree $\geq 0$ on such a component.

It remains to argue that if $x \in E$ then $\mathscr O_C(k\mu)$ has a section that does not vanish at $x$, at least for sufficiently large $k$.  In fact, we will find the required section when $k = 1$.  Since $\mu \geq \lambda$, we have an exact sequence:
\begin{equation*}
0 \to \mathscr O_C(\lambda) \to \mathscr O_C(\mu) \to \mathscr O_E(\delta) \to 0
\end{equation*}
	Pushing forward to $S$, using the isomorphism $\mathscr O_C(\lambda) \simeq \omega_{C/S}(\Sigma)$ (Lemma~\ref{lem:lambda-canonical}), and the vanishing of $R^1 \pi_\star \bigl( \omega_{C/S}(\Sigma) \bigr)$ (Corollary~\ref{cor:twisted-dualizing}), we get a surjection:
\begin{equation*}
\pi_\star \mathscr O_C(\mu) \to \pi_\star \mathscr O_E(\delta)
\end{equation*}
We can certainly find a neighborhood of $\pi(x)$ and a section of of $\pi_\star \mathscr O_E(\delta)$ that does not vanish at $x$, so the surjectivity implies the same applies to $\pi_\star \mathscr O_C(\mu)$.  This proves the semiampleness.

From $\pi$-semiampleness, we get the finite generation of the section ring of $\mathscr O_C(\mu)$ \cite[Example~2.1.30]{Lazarsfeld} and a proper, birational projection $\tau : C \to \overline C$.  From the triviality of $\mathscr O_C(\mu)$ on $E$, and the ampleness elsewhere, we see that the exceptional locus of $\tau$ is $E$.  For the remaining claims, which only concern the fibers of $\overline\pi$, we can assume that $S$ is the spectrum of a discrete valuation ring, since the by Lemma~\ref{boundary}, the construction commutes with base change to a discrete valuation ring satisfying the same hypotheses as $S$.

If the total space of $C$ is smooth at the points where $E$ meets the closure of $C \setminus E$ then we may apply Smyth's contraction lemma \cite[Lemma~2.13]{Smyth} to conclude.  It is possible to reduce to this case by replacing $C$ with a semistable model, but we will argue directly for clarity.

Now assuming that $S$ is the spectrum of a discrete valuation ring, note that $S$ is irreducible and normal. Moreover, $C$ is regular in codimension one (R1) since $C\to S$ has smooth generic fiber and has isolated singularities in fibers.  Since the fibers of $C_s$ over $S$ are reduced curves, they are (S2)~\cite[Remarques~IV.5.7.8]{EGA}.  Now $C\to S$ is flat, and $S$, being the spectrum of a discrete valuation ring, is certainly (S2).  Therefore the total space of $C$ is (S2)~\cite[Proposition~IV.6.8.3]{EGA}.  Since $C$ is smooth away from codimension $2$ in a neighborhood of $E$, it is (R1), and therefore $C$ satisfies Serre's criterion for normality near $E$.

We argue that $\overline C$ is reduced.  The components of $C_\Delta\setminus E$ map birationally to the components of $\overline C_\Delta$.  As $C_\Delta$ is reduced, $\overline C_\Delta$ is generically reduced.  On the other hand, flatness implies that the fiber $\overline C_\Delta$ is a Cartier divisor in $\overline C$, and is therefore (S1).  In particular, $\overline C_\Delta$ has no embedded points.  We conclude that $\overline C_\Delta$ is reduced.

The same argument we used on $C$ now implies that $\overline C$ is normal.  As $\tau$ certainly has connected fibers, and both $C$ and $\overline C$ are reduced, we obtain $\tau_\star \mathscr O_C = \mathscr O_{\overline C}$.

Furthermore, if $D$ is the closure of $C_\Delta \setminus E$ then $D$ is smooth at the points of $D \cap E$.  As $D \to \overline C_\Delta$ is birational, it follows that $D$ is the normalization of $\overline C_\Delta$ at $\phi(E)$.  This completes the proof of the third claim.

Finally, we verify that $\tau(E)$ is an elliptic $m$-fold point of $\overline C_\Delta$.  Since $C$ and $\overline C$ are generically isomorphic, they have the same arithemtic genus.  Therefore it suffices to show that $\overline C$ is Gorenstein.

Reduced fibers implies Cohen-Macaulay fibers, and any flat, projective, finitely presented morphism $C\to S$ whose geometric fibers are Cohen-Macaulay admits a relative dualizing sheaf \cite[Theorem 21]{kleiman1980relative} whose formation commutes with base change \cite[Proposition 9]{kleiman1980relative}, and the relative dualizing sheaf is (S2) \cite[Corollary 5.69]{kollar2008birational}.  It will therefore suffice to show that $\omega_{\overline C/S}$ is isomorphic to a line bundle in codimension one, since on a reduced scheme of finite type over a field (S2) sheaves isomorphic in codimension one are isomorphic \cite[Lemma 5.1.1]{FvdGL11}. To see this, note that 
\[
\mathscr{O}_{\overline C}(1)\rest{\overline C\setminus \tau(E)} \cong \omega_{\overline C/S}(\Sigma)\rest{\overline C\setminus \tau(E)}.
\]
Note $\tau(E)$ is the exceptional image and it is codimension 2, so this is an isomorphism in codimension one by definition. So we have shown that the relative dualizing sheaf on $\overline C$, which commutes with base extension, is isomorphic to a line bundle $\mathscr{O}_{\overline C}(1)$ near $\tau(E)$. In particular the fibers are Gorenstein curves. The fact that the fibers are stable in the sense of Smyth is immediate from our stability condition, so we have proved (3).
\end{proof}

\subsection*{Proof of Theorem \ref{thm: contraction to smyth}}
Now that we have developed the machinery for contracting a radially aligned log curve to an $m$-stable curve in the sense of Smyth, we finish the proof of Theorem \ref{thm: contraction to smyth}. 

\begin{proof}
We take $S = \Mbar^{\rm rad}_{1,n}$.  Let $\delta^m$ be as in Definition~\ref{def:delta_m}, and let $\widetilde C_m = C_{\delta^m}$ be as in Proposition~\ref{prop:C_delta}.  Note that $\mathcal O_S \to \mathcal O_S(\delta^m)$ is injective, because $\Mbar^{\rm rad}_{1,n}$ is logarithmically smooth.  We apply Proposition~\ref{contraction} to obtain a contraction $\widetilde C_m \to \overline C_m$.  As $\overline C_m$ is an $m$-stable curve in the sense of Smyth, this gives a map $\Mbar^{\mathrm{rad}}_{1,n} \to \Mbar_{1,n}(m)$.  When $\delta^m = 0$, the maps $C_m \to C$ and $C_m \to \overline C_m$ are isomorphisms, so our map is birational.
\end{proof}

\numberwithin{theorem}{section}
\section{The stable map spaces}\label{sec: vz-construction}

Let $Y$ be a variety over the complex numbers equipped with the trivial logarithmic structure. Let $\Mbar_{1,n}(Y,\beta)$ be the moduli space of stable $n$-pointed genus $1$ stable maps to $Y$, with image curve class $\beta$. By forgetting the map, we obtain a morphism
\[
\Mbar_{1,n}(Y,\beta)\to \fM_{1,n}
\]
to the stack of $n$-pointed prestable curves of genus $1$.  

Let $\fM_{1,n}^{\mathrm{rad}}$ be the moduli space of minimal families of radially aligned genus~$1$ logarithmic curves $\pi : C \rightarrow S$. We define $\widetilde{\VZ}_{1,n}(Y,\beta)$ to be the stack making the following diagram cartesian:
\[
\begin{tikzcd}
\widetilde{\VZ}_{1,n}(Y,\beta) \arrow{d}\arrow{r} & \Mbar_{1,n}(Y,\beta)\arrow{d} \\
\fM_{1,n}^{\mathrm{rad}} \arrow{r} & \mathfrak M_{1,n}.
\end{tikzcd}
\]
By definition $\widetilde{\VZ}_{1,n}(Y,\beta)$ parameterizes the following data over a logarithmic scheme $S$:
\begin{enumerate}
\item a logarithmic curve $C$ over $S$ having genus $1$ and $n$ marked points, together with a radial alignment of the tropicalizations;
\item a stable map $C \rightarrow Y$ of homology class $\beta$.
\end{enumerate}

Consider a family of maps from radially aligned curves over $S$, let $s$ be a geometric point of $S$. Denote by $\lambda$ the function on the vertices of the tropicalization $\plC_s$ of $C_s$ whose value on a vertex $v$ is the distance of $v$ from the circuit.  By assumption, the set of values $\lambda(v)$ is totally ordered.  Define the \textbf{contraction radius $\delta_s$} to be the smallest $\lambda(v)$, as $v$ ranges among the vertices of the dual graph of $C_s$, such that $f$ is non-constant on the corresponding component of $C_s$. In other words, $\delta_s$ measures the distance from the circuit to the closest non-contracted component.

Now suppose that $t \leadsto s$ is a geometric specialization.  Let $w$ be a component of $C_t$.  If $f$ is constant on all components $v$ of $C_s$ in the closure of $w$ then by the rigidity lemma \cite[Section~4, p.~43]{Mumford-AV}, $f$ is also constant on $w$.  Conversely, if $f$ is constant on $w$ then it is constant on all components of $C_s$ in the closure of $w$.  It follows that $\delta_t$ is the image of $\delta_s$ under the generization map $\overline M_{S,s} \rightarrow \overline M_{S,t}$.  Thus the collection of $\delta_s$ glues together into a section $\delta$ of $\overline M_S$ over $S$.

By Proposition~\ref{contraction}, the section $\delta$ induces a canonical logarithmic modification $\widetilde C \rightarrow C$ and contraction $\widetilde C \rightarrow \overline C$ over $S$, where $\overline C$ is a family of prestable curves in the sense of Smyth.  

We define $\VZ_{1,n}(Y,\beta)$ by imposing a closed condition on $\widetilde{\VZ}_{1,n}(Y,\beta)$:

\begin{definition}
Let $\VZ_{1,n}(Y,\beta)$ be the substack of $\widetilde{\VZ}_{1,n}(Y,\beta)$ parametrizing families of maps $C\to Y$, with notation as above, with the following \textbf{factorization property}: in the notation of the paragraph above, the composition $\widetilde C \rightarrow C \rightarrow Y$ factors through $\overline C$, in~\eqref{eqn:factorization}. 
\begin{equation} \label{eqn:factorization} \vcenter{\xymatrix{
\widetilde C \ar[r] \ar[d] & C \ar[d] \\
\overline C \ar[r] & Y
}} \end{equation}
\end{definition}

Note that the morphism $\overline{C}\to Y$ is by definition nonconstant on some branch of the component containing the genus~$1$ singularity.

Algebraicity is a consequence of general results applied to our framework.

\begin{lemma}\label{lem: algebraicty}
Suppose $Y$ is a quasiseparated algebraic space that is locally of finite presentation.  Then $\VZ_{1,n}(Y,\beta)$ is representable by algebraic spaces, locally of finite presentation, and quasiseparated over $\fM_{1,n}^{\mathrm{rad}}$.  If $Y$ is quasiprojective then $\VZ_{1,n}(Y,\beta)$ is locally quasiprojective over $\fM_{1,n}^{\mathrm{rad}}$.
\end{lemma}

\begin{proof}
For any $S$-point of $\fM_{1,n}^{\mathrm{rad}}$, we show that the fiber product $S \mathop\times_{\fM_{1,n}^{\mathrm{rad}}} \VZ_{1,n}(Y,\beta)$ has the requisite properties over $S$.  Over $S$, we have a diagram of curves
\begin{equation*} \xymatrix{
\widetilde C \ar[r] \ar[d] & C \\
\overline C
} \end{equation*}
that is constructed as was indicated above.  We can identify $S \mathop\times_{\fM_{1,n}^{\mathrm{rad}}} \VZ_{1,n}(Y,\beta)$ as the stable locus of a fiber product of $\Hom$-spaces over $S$,
\begin{equation*} \xymatrix{
\Hom_S(C, Y) \mathop\times_{\Hom_S(\overline C, Y)} \Hom_S(\widetilde C, Y) 
} \end{equation*}
As $C$, $\overline C$, and $\widetilde C$ are all flat, proper, and of finite presentation over $S$, we may apply \cite[Theorem~1.2]{Hall-Rydh} to obtain the algebraicity, finite presentation, and quasiseparatedness of the fiber product.  The stability condition cutting out $\VZ_{1,n}(Y,\beta)$ is open.  If $Y$ is quasiprojective then the $\Hom$-schemes are all quasiprojective \cite[Section 4.c]{FGA-IV}, so $\VZ_{1,n}(Y,\beta)$ is as well.
\end{proof}

The factorization property is satisfied by all limits of maps from smooth curves.

\begin{theorem}\label{vz-properness}
Assume that $Y$ is proper.  Then $\VZ_{1,n}(Y,\beta)$ is proper.
\end{theorem}
\begin{proof}
As it is pulled back from the modification $\fM_{1,n}^{\mathrm{rad}} \to \mathfrak M_{1,n}$, the moduli space $\widetilde{\VZ}_{1,n}(Y,\beta)$ is certainly proper over $\Mbar_{1,n}(Y,\beta)$.  We argue that the map $i : \VZ_{1,n}(Y,\beta) \to \widetilde{\VZ}_{1,n}(Y,\beta)$, which is a monomorphism by definition, is a closed embedding.  We will do this by showing $i$ is quasicompact and satisfies the valuative criterion for properness.  It is not necessary to check that $i$ is locally of finite type, as quasicompactness and the valuative criterion imply $i$ is universally closed \cite[Tag 01KF]{stacks-project}, and it is not difficult to deduce from this that $i$ is a closed embedding.

We begin with quasicompactness.  This is a local condition in the constructible topology on $\widetilde{\VZ}_{1,n}(Y,\beta)$ \cite[Proposition~(IV.1.9.15)]{EGA}, so we may replace $\widetilde{\VZ}_{1,n}(Y,\beta)$ with the components of any stratification into locally closed subsets $S$. 

An $S$-point of $\widetilde{\VZ}_{1,n}(Y,\beta)$ gives a morphism $f : \widetilde C \rightarrow Y$ and it lies in $\VZ_{1,n}(Y,\beta)$ if and only if $f$ factors through the contraction $\tau : \widetilde C \rightarrow \overline C$ by a morphism $g : \overline C \rightarrow Y$.  By the construction of $\tau$, we know that $f$ factors \emph{topologically} through $\tau$, so we obtain a homomorphism
\begin{equation*}
g^{-1} \mathscr O_Y \rightarrow \tau_\star \mathscr O_{\widetilde C} .
\end{equation*}
For $f$ to lie in $\VZ_{1,n}(Y,\beta)$ means precisely that the image of this homomrphism is contained in the subring $\mathscr O_{\overline C} \subset \tau_\star \mathscr O_{\widetilde C}$.  Now, the obstruction to factorization through $\mathscr O_{\overline C}$ is the composition
\begin{equation*}
\gamma : g^{-1} \mathscr O_Y \rightarrow \tau_\star \left(\mathscr O_{\widetilde C}\right) / \mathscr O_{\overline C} .
\end{equation*}
Replacing $S$ with a stratification, we can assume that the combinatorial types of $\widetilde C$ and $\overline C$ and the contraction $\tau$ are constant.  Under this assumption, the formation of $\tau_\star (\mathscr O_{\widetilde C}) / \mathscr O_{\overline C}$ commutes with base change in $S$.  Note that, because $\tau_\star \mathscr O_{\widetilde C_s}$ is the structure sheaf of the seminormalization of $\widetilde C_s$ when $s$ is a geometric point, the quotient $\tau_\star (\mathscr O_{\widetilde C_s}) / \mathscr O_{\overline C}$ has dimension either $0$ or $1$.
We can therefore identify the points $s$ of $S \mathop\times_{\widetilde{\VZ}_{1,n}(Y,\beta)} \VZ_{1,n}(Y,\beta)$ as those where $\tau_\star (\mathscr O_{\widetilde C_s}) / \mathscr O_{\overline C_s} = 0$ (which is an open subset) or where the cokernel of $\gamma_s$ is nonzero (which is closed).  
In any case, it is constructible.

Now we address the valuative criterion for properness.  Let $S$ be the spectrum of a valuation ring with generic point $\eta$.  Assume that $\eta$ has a logarithmic structure $M_\eta$.  We give $S$ the \emph{maximal} logarithmic structure extending $M_\eta$; that is, we set $M_S = \mathscr O_S \mathop\times_{\mathscr O_\eta} M_\eta$.  We assume that we already have a commutative diagram of solid lines
\begin{equation*} \xymatrix{
\widetilde C_\eta \ar[r] \ar[d]_\tau & \widetilde C \ar[d]_<>(0.3)\tau \ar[r]^-f & Y \\
\overline C_\eta \ar[urr] \ar[r]_j & \overline C  \ar@{-->}[ur]_g
} \end{equation*}
that we wish to extend by a dashed arrow.  By definition, $f$ factors topologically through $\overline C$, and does so uniquely, so we certainly have the horizontal arrow of the diagram below:
\begin{equation*} \xymatrix{
& \mathscr O_{\overline C} \ar[d]^{\varphi} \\
g^{-1} \mathscr O_Y \ar@{-->}[ur] \ar[r] & \displaystyle j_\star \mathscr O_{\overline C_\eta} \mathop\times_{j_\star \tau_\star \mathscr O_{\widetilde C_\eta}} \tau_\star \mathscr O_{\widetilde C}
} \end{equation*}
In order to promote $g$ to morphism of schemes, we must find a dashed arrow completing the diagram above.  We will do so by showing that $\varphi$ is an isomorphism.  We introduce the notation $\mathscr A = j_\star \mathscr O_{\overline C_\eta} \mathop\times_{j_\star \tau_\star \mathscr O_{\widetilde C_\eta}} \tau_\star \mathscr O_{\widetilde C}$.

Since $\widetilde C$ is flat over $S$, the sheaf $\mathscr O_{\widetilde C}$ is torsion free, and therefore $\tau_\star \mathscr O_{\widetilde C}$ is torsion free as well.  Thus, the subring $\mathscr A \subset \tau_\star \mathscr O_{\widetilde C}$ is also torsion free, and therefore flat over $S$ by \cite[Tag 0539]{stacks-project}.  

Observe now that the quotient $\mathscr A / \mathscr O_{\overline C}$ is finite over $S$, concentrated at the genus~$1$ singularity in the special fiber over $S$.  Therefore the exact sequence
\begin{equation*} 
0 \to \mathscr O_{\widetilde C} \xrightarrow{\varphi} \mathscr A \to \mathscr A / \mathscr O_{\widetilde C} \to 0
\end{equation*}
gives
\def\length{\operatorname{length}}%
\begin{equation*}
\chi(\mathscr A) = \chi(\mathscr O_{\overline C}) + \length(\mathscr A / \mathscr O_{\overline C}) .
\end{equation*}
But $\mathscr A$ and $\mathscr O_{\overline C}$ agree generically, and Euler characteristic is constant in flat families, so $\length(\mathscr A / \mathscr O_{\overline C})$ is $0$ and $\varphi : \mathscr O_{\overline C} \rightarrow \mathscr A$ is an isomorphism.  This proves the valuative criterion. Thus $\VZ_{1,n}(Y,\beta)$ is closed in $\widetilde{\VZ}_{1,n}(Y,\beta)$, and thus, proper.
\end{proof} 

\setcounter{subsection}{\value{theorem}}
\numberwithin{theorem}{subsection}
\subsection{Obstruction theory {\it \&} the virtual class} The standard construction for the virtual class of the Kontsevich space relative to the moduli space of curves applies to the moduli space $\VZ_{1,n}(Y,\beta)$. Let $\mathsf{vdim}$ denote the expected dimension of the moduli space of stable maps of genus $1$ to $Y$, i.e. 
%
\[
\mathsf{vdim} = -K_Y\cdot \beta+n,
\]
where $K_Y$ is the canonical class of $Y$.

\begin{theorem}\label{thm: vfc}
The moduli space $\VZ_{1,n}(Y,\beta)$ possesses a virtual fundamental class
\[
[\VZ_{1,n}(Y,\beta)]^{\mathrm{vir}}\in A_\mathsf{vdim}(\VZ_{1,n}(Y,\beta)).
\]
\end{theorem}

\begin{proof}
Consider the forgetful morphism
\[
\pi: \VZ_{1,n}(Y,\beta)\to \mathscr \fM_{1,n}^{\mathrm{rad}}.
\]
By well-known deformation theory for morphisms from curves to smooth targets, there exists a relative perfect obstruction theory 
\[
E^\bullet \to \mathbf{L}^\bullet_{\VZ_{1,n}(Y,\beta)/\fM_{1,n}^{\mathrm{rad}}}
\]
with $E^\bullet = R \pi_\star (f^\star T_Y)^\vee$.  The complex $E^\bullet$ determines a vector bundle stack $\mathbf E$ over the moduli space $\VZ_{1,n}(Y,\beta)$ the map $\pi$ has Deligne--Mumford type, in the sense of~\cite[Section 2]{Mano12}. Applying Manolache's virtual pullback $\pi_{\mathbf E}^{!}$ to the fundamental class of  $\fM_{1,n}^{\mathrm{rad}}$, we obtain a virtual fundamental class in expected dimension.
\end{proof}

\subsection{Maps to projective space} \label{sec:target-proj}
The main result of this section is the smoothness of the space of maps to $\PP^r$. 

\begin{theorem}\label{thm: vz-smoothness}
The moduli space $\VZ_{1,n}(\PP^r,d)$ is smooth of dimension 
\[
\dim \ \VZ_{1,n}(\PP^r,d) = (r+1)d+n,
\]
and its virtual fundamental class is equal to the usual fundamental class.
\end{theorem}

We begin with a lemma that is more general than we need at this stage, but will be useful when we consider quasimaps in the sequel.

\begin{lemma} \label{lem:h1'}
Let $C$ be a Gorenstein curve of genus~$1$ and let $L$ be a line bundle on $C$ that has degree $\geq 0$ on all components and positive degree on at least one component of the circuit of $C$.  Then $H^1(C, L) = 0$.
\end{lemma}

\begin{proof}
Let $C_0$ be the circuit component of $C$.  Then $H^1(C,L) = H^1(C_0,L_0)$, where $L_0$ denotes the restriction of $L$ to $C_0$.  The dualizing sheaf of $C_0$ is trivial (Proposition~\ref{prop:dualizing-trivial}), so $H^1(C,L)$ is dual to $H^0(C_0,L_0^\vee)$, which vanishes because $L_0^\vee$ has negative degree on at least one component of $C_0$ and degree $\leq 0$ on all other components.
\end{proof}

\begin{proof}[Proof of Theorem~\ref{thm: vz-smoothness}]
We will show that the map
\[
\pi: \VZ_{1,n}(\PP^r,d)\to \fM_{1,n}^{\mathrm{rad}}
\]
	is relatively unobstructed, and in fact that the map to the universal Picard stack is unobstructed.  The theorem will then follow from the smoothness of $\fM_{1,n}^{\mathrm{rad}}$ proved in Corollary~\ref{cor:smooth}.  Consider a lifting problem
\begin{equation*} \xymatrix{
S \ar[r] \ar[d] & \VZ_{1,n}(\PP^r,d) \ar[d] \\
S' \ar[r] \ar@{-->}[ur] & \fM_{1,n}^{\mathrm{rad}}
} \end{equation*}
in which $S'$ is a square-zero extension of $S$.  We view these data as a (minimal) radially aligned curve $C'$ over $S'$ restricting to $C$ over $S$ and a map $\overline C \to \PP^r$ that is nonconstant on at least one branch of the singular point of each fiber, and nonconstant on the genus~$1$ component when there is no singular point.  The map to $\PP^r$ can be seen as a line bundle $L$ on $\overline C$ with $r+1$ sections.  There is no obstruction to deforming $L$ to a line bundle $L'$ on $\overline C'$: obstructions lie in $H^2(\overline C, L)$.  The obstruction to deforming the sections is in $H^1(\overline C, L)$, which vanishes (locally in $S$) by Lemma~\ref{lem:h1'}, since $\overline C \to \PP^r$ is nonconstant on at least one branch of the singular point of each fiber.
\end{proof}

\begin{remark}
The proof shows  that $\VZ_{1,n}(\PP^r,d)$ is smooth and unobstructed relative to the universal Picard stack over $\fM_{1,n}^{\rm rad}$, since there is no restriction on the deformation of the line bundle used to deform the map.
\end{remark}

\subsection{The Vakil--Zinger blowup construction}\label{vz-comparison} 

In this section, we give a modular interpretation of Vakil and Zinger's blowup construction. This requires a mild variation of our moduli problem, where we replace \textbf{radial} alignment curves with the slightly more refined notion of \textbf{central} alignment. We begin with a review of Vakil and Zinger's construction. 

\numberwithin{theorem}{subsubsection}
\subsubsection{Vakil and Zinger's blowups}
Let $\fM_{1,n}$ be the moduli stack of $n$-pointed, genus~$1$ prestable curves.  For each geometric point $s$ of $\fM_{1,n}$, we write $\plC_s$ for the tropicalization of the corresponding curve.

Suppose that $\plC$ is a tropical curve of genus~$1$.  By a \textbf{precontractible tropical subcurve} or a \textbf{precontractible subcurve} for short, we will mean a subgraph $\plC^\circ \subset \plC$ that is either empty or such that
\begin{enumerate}
\item $\plC^\circ$ has genus~$1$, 
\item if $v\in \plC^\circ$, then any half-edge incident to $v$ is contained in $\plC^\circ$, and
\item the marking function on $\plC^\circ$ is the restriction of the marking function on $\plC$.
\end{enumerate}

We will think of the precontractible subcurve $\plC^\circ$ as being the information of a ``would be'' contracted subcurve. Let $\fM_{1,n}^{\dagger}$ denote the moduli space of of nodal $n$-pointed genus $1$ curves together with the additional information of a precontractible subgraph $\plC_s^\circ \subset \plC_s$ at each geometric point, such that, if $t \leadsto s$ is a geometric specialization then the complement of $\plC^\circ_s$ maps onto the complement of $\plC^\circ_t$.  In other words, a component that is not  ``contracted'' generizes to a component that is not formally ``contracted''.

\begin{figure}

\begin{minipage}[t]{.45\textwidth}
\begin{tikzpicture}

\begin{scope}[shift={(-5,0)}]

\coordinate (root) at (0,.75);
\coordinate (A) at (2,-1);
\coordinate (B) at (-2,-1);
\node (C) at (3.5,-4) {$5$};
\coordinate (D) at (1.5,-2);
\node (E) at (2,-4) {$4$};
\node (F) at (1,-4) {$3$};
\node (G) at (-3,-4) {$1$};
\node (H) at (-1,-4) {$2$};

\draw[color=gray!50,dashed,thick] (root) -- (A); \node at (1.25,0) {$\beta$};
\draw[color=gray!50,dashed,thick] (root) -- (B); \node at (-1.25,0) {$\alpha$};
\draw[color=gray!50,thick] (A) -- (C); 
\draw[color=gray!50,thick] (A) -- (D); \node at (1.5,-1.4) {$\gamma$};
\draw[color=black,thick] (D) -- (E);
\draw[color=black,thick] (D) -- (F);
\draw[color=black,thick] (B) -- (G);
\draw[color=black,thick] (B) -- (H);

\draw[color=black,fill=white] (root) circle (1mm);
\draw[color=gray!50,fill=gray!50] (A) circle (.75mm);
\draw[color=black,fill=black] (B) circle (.75mm);
\draw[color=black,fill=black] (D) circle (.75mm);

\end{scope}
\end{tikzpicture}

\captionsetup{width=\linewidth}

\caption{A genus $1$ graph containing a precontractible subgraph shown in gray and a smaller precontractible subgraph shown in dashed gray.  The smaller precontractible subgraph has $k = 2$ and $J = \emptyset$; the larger one has $k = 2$ and $J = \{ 5 \}$.  As usual, the open circle represents a vertex of genus~$1$ or a ring of genus~$0$ vertices.}

\label{fig:precontractible}

\end{minipage} \qquad
\begin{minipage}[t]{.45\textwidth}

\begin{tikzpicture}
\begin{scope}[scale=2.5,shift={(1.5,0)}]

\coordinate (A) at (.85,-.5);
\coordinate (B) at (-.85,-.5);
\coordinate (C) at (0,1);

\draw[color=black] (A) -- (B) -- (C) -- (A);
\draw[color=gray!50,dashed,thick] (-.425,.25) -- (A);
\draw[color=gray!50,thick] (-.425,.25) -- (.425,.25);

\draw[color=black,fill=black] (A) circle (.3mm); \node at (1.1,-.6) {$\gamma=1$};
\draw[color=black,fill=black] (B) circle (.3mm); \node at (-1.1,-.6) {$\beta=1$};
\draw[color=black,fill=black] (C) circle (.3mm); \node at (0,1.15) {$\alpha=1$};

\end{scope}

\end{tikzpicture}

\captionsetup{width=\linewidth}

\caption{Barycentric coordinates on the tropicalization of the deformation space of the tropical curve in Figure~\ref{fig:precontractible} and the subdivision induced by blowing up $\Upsilon(2,\emptyset)$ followed by the proper transform of $\Upsilon(2,\{ 5 \})$.}

\end{minipage}

\end{figure}

The definition of $\fM_{1,n}^\dagger$ realizes it as an \'etale sheaf over $\fM_{1,n}$, and $\fM_{1,n}^\dagger$ is representable by the espace \'etal\'e of that sheaf.  In particular, $\fM_{1,n}^\dagger$ is an algebraic stack and there is a projection map
\[
\fM_{1,n}^{\dagger}\to \fM_{1,n}
\]
that is \'etale but not separated.

The morphism $\Mbar_{1,n}(\PP^r,d)\to \fM_{1,n}$ can be factored through $\fM_{1,n}^\dagger$ by formally declaring components of a family $[f: C\to \PP^r]$ to be ``contracted'' when they are contracted by $f$, so we have
\[
\Mbar_{1,n}(\PP^r,d)\to \fM_{1,n}^\dagger.
\]

\begin{construction} \label{cons:upsilon}
Fix a non-negative integer $k$ and a subset $J \subset \{ 1, \ldots, n \}$.  By a \textbf{$(k,J)$-graph}, we will mean a tropical curve with a single vertex, of genus~$1$, and $k + |J|$ legs, with $|J|$ of them marked by the set $J$.

We write $\Upsilon(k,J) \subset \fM_{1,n}^\dagger$ for the locus of curves $C$ with tropicalization $\plC$ such that the subgraph marked for contraction $\plC^\circ \subset \plC$ has a precontractible subcurve with a weighted edge contraction onto a $(k,J)$-graph. This locus is a closed substack, as it is a union of closed strata in the stratification of $\fM_{1,n}^\dagger$ induced by its normal crossings boundary divisor. See~\cite[Section 2.6]{HL10} and~\cite[Section 1.2]{VZ08} for the corresponding loci in those setups.

Define a partial order
\[
(k',J')\preccurlyeq (k,J),
\]
if the strata are not equal, $k'\leq k$ and $J_E'\subset J_E$, and write $(k',J') \prec (k,J)$ to mean that at least one of these relations is strict. Choose any total ordering on the strata $\{\Upsilon(k,J)\}$ extending the partial order above. Let $\widetilde{\fM}_{1,n}^{\dagger}$ be the iterated blowup of $\fM_{1,n}^{\dagger}$ along the proper transforms of the loci $\Upsilon(k,J)$ in the order specified by the total order. It is part of~\cite[Theorem 1.1]{VZ08} that the resulting space is insensitive to the choice of total order extending $\preccurlyeq$. Note that each connected component of the stack $\widetilde{\fM}_1^{\dagger}$ is of finite type where only finitely many of the loci $\Upsilon(k,J)$ are non-empty, so the limit of this procedure is well-defined, as an algebraic stack. Using the morphism
\[
\Mbar_{1,n}(\PP^r,d)\to \fM_{1,n}^{\dagger},
\]
define the stack $\widehat{\cM}_1(\PP^r,d)$ as the proper transform
\[
\widehat{\cM}_{1,n}(\PP^r,d) := \Mbar_{1,n}(\PP^r,d)\times_{\fM_{1,n}^{\dagger}} \widetilde{\fM}_{1,n}^{\dagger}.
\]
Then the \textbf{Vakil--Zinger desingularization} of the main component of $\Mbar_{1,n}(\PP^r,d)$ is defined as the closure
\[
\widetilde{\cM}_1(\PP^r,d):= \overline{\left\{[f: C\to \PP^r]: \textrm{$C$ is a smooth curve of genus $1$}\right\}}
\]
inside $\widehat M_{1,n}(\PP^r, d)$.
\end{construction}

\subsubsection{Centrally aligned curves}
In Section~\ref{sec:aligned}, we introduced radial alignment as the datum necessary to contract a genus~$1$ component of a logarithmic curve $C$.  It is actually possible to construct a contraction with strictly less information. 

All that is really necessary is a radius dividing the tropicalization of $C$ into an interior, to be contracted, and an exterior, without the imposition of order between the individual vertices.  This leads to a logarithmically smooth, but non-smooth modification of the moduli space of curves~\cite{KSP-Thesis}, but the singularities can be resolved by ordering just the vertices of the interior.  To first approximation, this is the notion of a central alignment.

\begin{definition}
Let $C$ be a genus~$1$ logarithmic curve over $S$ with tropicalization $\plC$.  A \textbf{central alignment} of $C$ is the choice of $\delta \in \overline M_S$ such that
\begin{enumerate}
\item $\delta$ is comparable to $\lambda(v)$ for all vertices $v$ of $\plC$, and
\item the interior of the circle of radius $\delta$ around the circuit of $\plC$ is radially aligned.
\end{enumerate}
A central alignment on a family of curves over $S$ is a section of $\overline M_S$ that gives a central alignment of each geometric fiber.

If $\delta = \lambda(v)$ for at least one vertex $v$ of $\plC$ and the subgraph of $\plC$ where $\lambda < \delta$ is a stable curve then we call the central alignment \textbf{stable}.  A family of central alignments is stable if each of its fibers is stable.

We write $\mathfrak M_{1,n}^\ctr$ for the space of logarithmic curves of genus~$1$ with~$n$ markings and a stable central alignment.
\end{definition}

\begin{proposition}
$\mathfrak M_{1,n}^\ctr$ is a logarithmic modification of $\mathfrak M_{1,n}^\dagger$, and in particular is representable by an algebraic stack with a logarithmic structure and is logarithmically smooth.
\end{proposition}
\begin{proof}
We have a map $\mathfrak M_{1,n}^\ctr$ by declaring formally that the interior of the circle of radius $\delta$ is ``contracted''.  Then the rest of the proof of algebraicity is the same as that of Proposition~\ref{prop:modification}.  Logarithmic smoothness follows because it is logarithmically \'etale over the logarithmically smooth stack $\mathfrak M_{1,n}$.
\end{proof}

\begin{remark}
If the first part of the definition of a stable central alignment is omitted then the value $\delta$ can introduce a new parameter to the logarithmic structure of the moduli space.  Scaling this parameter gives a continuous family of automorphisms.
\end{remark}

\subsubsection{Comparing the constructions}

\begin{proposition} \label{prop:blowup-alignment}
The Vakil--Zinger blowup $\widetilde{\mathfrak M}_{1,n}^\dagger$ is the moduli space $\mathfrak M_{1,n}^\ctr$ of central alignments on logarithmic curves of genus~$1$.
\end{proposition}

\begin{proof}
The Vakil--Zinger blowups are logarithmic blowups, and therefore are equivalent to imposing order relations in the characteristic monoid $\overline M_S$ (see Section~\ref{sec:log-blowup}). Said differently, viewing $\overline M_S$ as the set of positive elements of the partially ordered group $\overline M_S^{\rm gp}$, the blowup is equivalent to refining this partial order. It follows that the Vakil--Zinger blowup $\widetilde{\mathfrak M}_{1,n}^\dagger$ represents a logarithmic \emph{subfunctor} of $\mathfrak M_{1,n}^\dagger$.  We show that the order imposed on the characteristic monoid by a stable central alignment is the same as the order imposed by the Vakil--Zinger blowups.

Because the sheaf of characteristic monoids is constructible, this is a pointwise assertion.  We must therefore prove that, if $S$ is the spectrum of an algebraically closed field, equipped with a logarithmic structure, then an $S$-point $[C]$ of $\mathfrak M_{1,n}^\dagger$ lies in $\widetilde{\mathfrak M}_{1,n}^\dagger(S)$ if and only if it lies in $\mathfrak M_{1,n}^\ctr(S)$.

Assume first that $[C]$ lies in $\mathfrak M_{1,n}^\ctr(S)$.  Let $\plC$ be the tropicalization of $C$ and let $\plC^\circ$ be the induced subgraph on the vertices $v$ such that $\lambda(v) < \delta$, equipped with the restriction of the marking, length, and genus functions.  We write $\widetilde\Upsilon(k,J)$ for the pullback of $\Upsilon(k,J)$ to $S$.

By definition of a central alignment, the vertices $v$ of $\plC^\circ$ are totally ordered by the lengths $\lambda(v)$.  Each $\lambda(v)$ therefore determines a circle on $\plC$, which crosses $k(v)$ finite edges of $\plC$ and $J(v)$ infinite legs.  We observe that, as $[C]$ lies in $\widetilde\Upsilon(k,J)$ if and only if the interior of the circle of radius $\lambda(v)$ has a weighted edge contraction onto a $(k,J)$-curve, this can occur only if $(k,J) = (k(v), J(v))$ for some vertex $v$ of $\plC^\circ$.

Blowing up $\widetilde\Upsilon(k(v),J(v))$ has the effect of requiring a minimum $\lambda(w)$ among the vertices $w$ of $\plC$ immediately outside the circle of radius $\lambda(v)$.  Since the vertices of $\plC^\circ$ are totally ordered by definition, and there is at least one vertex $w$ immediately outside of $\plC^\circ$ with $\lambda(w) = \delta$, we find that $[C]$ is contained in the blowup of $\widetilde\Upsilon(k(v),J(v))$, as required.

Now we prove that sequentially blowing up the $\Upsilon(k,J)$ imposes a central alignment.  Suppose that $[C]$ is an $S$-point of $\widetilde{\mathfrak M}_{1,n}^\dagger$, let $\plC$ be the tropicalization of $C$, and let $\plC^\circ$ be the formally contracted subgraph.  Write $\plC_0^\circ$ circuit of $\plC^\circ$, with the induced marking function.  Then, by contracting the circuit, $\plC_0^\circ$ contracts onto a $(k,J)$-graph.  Therefore $[C]$ lies in $\widetilde\Upsilon(k,J)$.

Since $[C]$ lies in $\widetilde{\mathfrak M}_{1,n}^\dagger$, the locus $\widetilde\Upsilon(k,J)$ has been blown up.  By definition of the logarithmic blowup (see Section~\ref{sec:log-blowup}), this means that there is a vertex of $\plC$ on the periphery of $\plC^\circ_0$ that is minimal with respect to $\lambda$.  We call this vertex $v_0$.

Now we proceed by induction.  Assume that we have already found vertices $v_0, v_1, \ldots, v_i$ such that $v_j$ is minimal among the vertices of $\plC^\circ$, excluding $v_0, \ldots, v_{j-1}$.  Then the circle of radius $\lambda(v_i)$ crosses $\plC$ at $k(v_i)$ edges and $J(v_i)$ legs.  Therefore $[C]$ is contained in $\widetilde\Upsilon(k(v_i), J(v_i))$.  

Exactly as in the base case, $\widetilde\Upsilon(k(v_i),J(v_i))$ has been blown up, so there is a $v_{i+1}$ in the periphery of $\plC^\circ_i$ such that $\lambda(v_i)$ is minimal.  The induction proceeds until we run out of vertices in $\plC^\circ$ and the vertices are therefore totally ordered.
\end{proof}

For proper $Y$, we may now define a stack $\widetilde \VZ\vphantom{\VZ}^\ctr_{1,n}(Y,\beta)$ of stable maps from the universal centrally aligned curve to $X$, via a fiber product:
\[
\begin{tikzcd}
\widetilde{\VZ}\vphantom{\VZ}_{1,n}^\ctr(Y,\beta) \arrow{d} \arrow{r} & \Mbar_{1,n}(Y,\beta) \arrow{d}\\
\fM_{1,n}^\ctr \arrow{r} & \fM_{1,n}.
\end{tikzcd}
\]

Just as in Section~\ref{sec:target-proj}, given a map from a centrally aligned curve $[f: C\to Y]$ over a logarithmic scheme $S$, we obtain a radius $\delta_f$, which is the distance from the genus $1$ contracted component to the closest non-contracted component of $C$, and thus a contracted curve $\widetilde C\to \overline C$ from a partial destabilization of $C$. We define the stack $\VZ^\ctr_{1,n}(Y,\beta)$ to be the locus of maps satisfying the \textbf{factorization property}, as before.  The proofs of smoothness and properness go through exactly as in Section~\ref{sec:target-proj}.

\begin{theorem} \label{thm:vz-text}
There is an isomorphism between the Vakil--Zinger blowup with the moduli space of centrally aligned maps to $\PP^r$
\[
\VZ^\ctr_{1,n}(\PP^r,d)\to \widetilde{\cM}_{1,n}(\PP^r,d)
\]
that commutes with the projection to $\Mbar(\PP^r,d)$.
\end{theorem}
\begin{proof}
By definition, $\widetilde\cM_{1,n}(\PP^r,d)$ is the closure of the main component of the space of maps from the universal curve over $\widetilde\fM_{1,n}^\dagger$ to $\PP^r$.  But we saw in Proposition~\ref{prop:blowup-alignment} that $\widetilde\fM_{1,n}^\dagger$ is isomorphic to $\fM_{1,n}^\ctr$, so $\widetilde\cM_{1,n}(\PP^r, d)$ is the closure of the main component of $\widetilde\VZ\vphantom{\VZ}_{1,n}^\ctr(\PP^r,d)$.  On the other hand, $\VZ_{1,n}^\ctr(\PP^r,d)$ is a smooth, proper, and connected substack of $\widetilde\VZ\vphantom{\VZ}_{1,n}^\ctr(\PP^r,d)$ that contains the main component.  Hence it coincides with $\widetilde{\cM}_{1,n}(\PP^r,d)$.
\end{proof}

\begin{remark}
We could have chosen to work with centrally aligned logarithmic curves throughout the paper. However, there are some advantages to radially aligned curves. One obtains a \textbf{single} moduli space $\Mbar_{1,n}^{\mathrm{rad}}$ which maps to all the spaces of Smyth curves. The discussion of logarithmic targets in the sequel to this paper is also be cleaner with a radial alignment. On the other hand, the advantage of the Vakil--Zinger approach and central alignments is that fewer blowups are required, and the locus of maps where no elliptic component is contracted remains untouched by the construction.  Vakil--Zinger could have just as easily produced a blowup construction of $\VZ_{1,n}(\PP^r,d)$ by blowing up more loci than was strictly necessary for smoothness. 
\end{remark}

\numberwithin{theorem}{section}
\section{The quasimap spaces}\label{sec: quasimaps} 

A modification of the methods of the previous sections gives rise to a desingularization of the genus $1$ quasimaps spaces to $\PP^r$, constructed by Ciocan-Fontanine and Kim~\cite{CFK} and Marian, Oprea, and Pandharipande~\cite{MOP}. 

\begin{definition}
A \textbf{genus $g$ quasimap to $\PP^r$ over $S$} consists of the data
\[
((\mathscr C,p_1,\ldots, p_n),\mathscr L,s_0,\ldots, s_r),
\]
where $(\mathscr C,p_1,\ldots, p_n)\to S$ is a flat family of $n$-pointed nodal curves of genus $g$, $\mathscr L$ is a line bundle on $\mathscr C$ with sections $s_0,\ldots, s_r$, such that on every geometric fiber $C$ of $\mathscr C$, the following non-degeneracy condition holds: \textit{there is a finite (possibly empty) set of non-singular unmarked points $B$ of $C$, such that, outside $B$ the sections $s_0,\ldots,s_r$ are basepoint free.}
\end{definition}

Such a quasimap determines a homomorphism
\[
\mathrm{Pic}(\PP^r)\to \mathrm{Pic}(C),
\]
and via Poincar\'e duality, a homology class in $H_2(\PP^r,\ZZ)$. We refer to this as the \textbf{degree} of the quasimap. An isomorphism of quasimaps is defined in the natural fashion, as an isomorphism of two families of curves $\mathscr C_1\to \mathscr C_2$, with compatible isomorphisms of the pullbacks of the line bundle and sections of the latter with those of the former.

\begin{definition}
A quasimap $((\mathscr C,p_1,\ldots, p_n),\mathscr L,s_0,\ldots, s_r)$ is said to be \textbf{stable} if
\[
\omega_{\mathscr C/S}(p_1+\cdots+ p_n)\otimes \mathscr L
\]
is ample. 
\end{definition}

As asserted in~\cite{CFK}, this is equivalent to a combinatorial condition on each geometric fiber: (1) no rational component of the the underlying curve $C$ of the quasimap can have fewer than $2$ special points (nodes and markings), and (2) on every rational component with $2$ special points, or elliptic component with $1$ special point, the line bundle $\mathscr L$ must have positive degree. 

\begin{theorem}[{\cite{CFK,MOP}}]
There is a Deligne-Mumford stack $\mathcal Q_{g,n}(\PP^r,d)$ parametrizing stable quasimaps of genus $g$ with $n$-marked points to $\PP^r$ of degree $d$. Moreover, the natural map to the universal Picard variety
\[
\mathcal Q_{g,n}(\PP^r,d)\to \mathfrak{Pic}_{g,n}
\]
defines a relative perfect obstruction theory on $\mathcal Q_{g,n}(\PP^r,d)$ and thus a virtual fundamental class.
\end{theorem}

Here, $\mathfrak{Pic}_{g,n}$ denotes the moduli stack paramterizing pairs of nodal $n$-marked genus $g$ curves together with a line bundle.

When $g = 1$ and $n = 0$, these spaces exhibit a remarkable smoothness property~\cite[Section 3.3]{MOP}:

\begin{theorem}
The moduli stack $\mathcal Q_{1,0}(\PP^r,d)$ is smooth.
\end{theorem}

It should be noted that this property fails as soon as there are marked points. The smoothness is due to the strength of the stability condition in the quasimaps theory.  Without marked points, rational tails are disallowed, and thus, no genus $1$ curve can be contracted. Our construction in the stable maps case can be adapted to desingularize the moduli spaces $\mathcal Q_{1,n}(\PP^r,d)$ for $n>0$. 

As in the stable maps case, given a line bundle on a family of radially aligned curves $\mathscr L$ on $C\to S$, at each geometric point $s\in S$, there is a well-defined contracting radius $\delta_s$, measuring the distance from the circuit to the first component on which $\mathscr L$ has nonzero degree. This defines a destabilization $\widetilde C\to C$ and a contraction $\widetilde C\to \overline C$. 

\begin{definition}
Define the stack $\widetilde{\VQ}_{1,n}(\PP^r,d)$ as the stack parametrizing a minimal radially aligned logarithmic curve $C\to S$ of genus $1$ and a quasimap on $C$. \\

\noindent
Define the stack ${\VQ}_{1,n}(\PP^r,d)$ as the substack of $\widetilde{\VQ}_{1,n}(\PP^r,d)$ parametrizing stable quasimaps 
\[
((\mathscr C,p_1,\ldots, p_n),\mathscr L,s_0,\ldots, s_r)
\]
with the following \textbf{factorization property}: In the notation of the previous section, let $\tau: \widetilde C\to C$ and $\gamma: \widetilde C\to \overline C$ be the partial destabilization and Gorenstein contraction of $C$. Then, there is a line bundle $\overline {\mathscr L}$ on $\overline C$ with sections $\{\overline s_i\}_{i=0}^r$such that 
\[
\tau^\star {\mathscr L} = \gamma^\star \overline {\mathscr L},
\]
the sections $\tau^\star s_i$ coincide with $\gamma^\star \overline s_i$. \end{definition}

\begin{theorem}
The stack $\VQ_{1,n}(\PP^r,d)$ is a smooth and proper Deligne--Mumford stack.
\end{theorem}

We separate the proof into three lemmas.  The algebraicity is proved in Lemma~\ref{lem:vq-alg}, the smoothness in Lemma~\ref{lem:vq-smooth}, and the properness in Lemma~\ref{lem:vq-proper}.

\begin{lemma} \label{lem:vq-alg}
$\VQ_{1,n}(\PP^r,d)$ is a Deligne--Mumford stack.
\end{lemma}
\begin{proof}
Algebraicity follows from the same arguments as Lemma~\ref{lem: algebraicty}, replacing the Hom-stack of maps to $\PP^r$ with maps to $[\mathbf A^{r+1} / \Gm]$, which is algebraic by~\cite[Theorem~1.2]{Hall-Rydh}, noting that stability is an open condition.
\end{proof}

\begin{lemma} \label{lem:vq-smooth}
$\VQ_{1,n}(\PP^r,d)$ is smooth over the universal Picard stack over $\mathfrak M_{1,n}$.
\end{lemma}
\begin{proof}
Once again, the key fact is that $\overline {\mathscr L}$ has positive degree on at least one branch of the component containing the genus $1$ singularity. Let $ {\mathscr U}\to \fM_{1,n}^{\mathrm{rad}}$ be the universal radially aligned curve.  Let $\mathfrak{Pic}({\mathscr U})$ be the relative Picard scheme over this curve. Note that $\mathfrak{Pic}({\mathscr U})$ is smooth over a smooth base, since obstructions to deforming line bundles on a curve $C$ lie in $H^2(C, \mathscr O_C)$, and vanish for dimension reasons. To prove smoothness of $\VQ_{1,n}(\PP^r,d)$ it suffices to show that the relative obstructions of the map
\[
\VQ_{1,n}(\PP^r,d) \to \mathfrak{Pic}({\mathscr U}).
\]
vanish. Let $(C,\overline C,L,\{s_i\})$ be a quasimap from a radially aligned curve, with the factorization property as described above. Fixing a deformation of the curve and line bundle $(C,L)$, the deformations of the sections are obstructed by $H^1(\overline C,\mathscr L)$. These obstructions were already shown to vanish in Lemma~\ref{lem:h1'}. 
\end{proof}

\begin{lemma} \label{lem:vq-proper}
$\VQ_{1,n}(\PP^r,d)$ is closed in $\mathcal Q_{1,n}(\PP^r,d)$.
\end{lemma}
\begin{proof}
Since $\VQ_{1,n}(\PP^r,d) \to \mathcal Q_{1,n}(\PP^r,d)$ is a monomorphism, it is sufficient to verify the valuative criterion.  Assume that $S$ is the spectrum of a valuation ring with generic point $j : \eta \to S$, and the maximal extension $M_S$ of a logarithmic structure $M_\eta$ on $\eta$, we want to lift a diagram~\eqref{eqn:val-lift}:
\begin{equation} \label{eqn:val-lift} \vcenter{\xymatrix{
\eta \ar[r] \ar[d] & \VQ_{1,n}(\PP^r, d) \ar[d] \\
S \ar[r] \ar@{-->}[ur] & \mathcal Q_{1,n}(\PP^r, d)
}} \end{equation}
The map $S \to \mathcal Q_{1,n}(\PP^r,d)$ gives a family, $C$, of logarithmic genus~$1$ curves over $S$, and a stable quasimap $(L, x_0, \ldots, x_n)$ on $C$.  The map $\eta \to \VQ_{1,n}(\PP^r, d)$ gives a radial alignment on $C_\eta$, which extends uniquely to $C$ by the properness of the space of radially aligned curves.  The quasimap $(L, x_0, \ldots, x_n)$ induces a contraction radius $\delta \in \Gamma(S, \overline M_S)$, which provides a destabilization $\upsilon : \widetilde C \to C$ and a contraction $\tau : \widetilde C \to \overline C$, all over $S$.

By assumption, $\upsilon^\star (L, x_0, \ldots, x_n) \big|_\eta$ descends along $\tau$ to a stable quasimap $(\overline L_\eta, \overline x_0, \ldots, \overline x_n)$ on $\overline C_\eta$.  We wish to show that $(L, x_0, \ldots, x_n)$ descends to $\overline C$.

Let $E$ be the interior of the contraction radius inside $\widetilde C$ --- the locus contracted by $\tau$.  By definition the contraction radius, $L$ has degree zero on all components of the fibers of $E$.  But $x_0, \ldots, x_n$ are sections of $L$ that do not vanish identically on any component of any fiber of $\widetilde C$ over $S$.  Therefore, $\upsilon^\star L \big|_E$ is trivialized by at least one of the $x_i$.

Now, let 
\begin{equation*}
\overline L = j_\star \overline L_\eta \mathop\times_{j_\star \tau_\star \upsilon^\star L_\eta} \tau_\star \upsilon^\star L .
\end{equation*}
As the map
\begin{equation} \label{eqn:pushforward-iso}
\mathscr O_{\overline C} \to j_\star \mathscr O_{\overline C_\eta} \mathop\times_{j_\star \tau_\star \mathscr O_{\widetilde C_\eta}} \tau_\star \mathscr O_{\widetilde C}
\end{equation}
is an isomorphism (see the proof of Theorem~\ref{vz-properness}), and $L$ can be trivialized in a neighborhood of $E$, the sheaf $\overline L$ is invertible on $\overline C$.  Moreover, there is a natural map $\tau^\star \overline L \to \upsilon^\star L$ which is an isomorphism away from $E$, since $\tau$ is an isomorphism there, and an isomorphism near $E$, by the isomorphism~\eqref{eqn:pushforward-iso}.

The sections $x_0, \ldots, x_n$ descend automatically to $\overline L$, so the proof of the valuative criterion, and of the lemma, is complete.
\end{proof}


\begin{remark}
One can construct $\VQ_{1,n}(\PP^r,d)$ as a blowup of $\mathcal Q_{1,n}(\PP^r,d)$ in analogous fashion to Vakil and Zinger's desingularization of Kontsevich space, sequentially blowing up the loci of quasimaps that have degree $0$ on a curve of arithmetic genus $1$, to arrive at the moduli space above. Also, as in the stable maps case, there is a centrally aligned variant where the blowups are done in a slightly more efficient fashion.
\end{remark}

\bibliographystyle{siam} 
\bibliography{EllipticStableMaps}

\end{document}